\documentclass{amsart}

\input xy
\xyoption{all}

\newbox\noforkbox \newdimen\forklinewidth
\setbox0\hbox{$\textstyle\smile$}
\setbox1\hbox to \wd0{\hfil\vrule width \forklinewidth depth-2pt
 height 10pt \hfil}
\setbox\noforkbox\hbox{\lower 2pt\box1\lower 2pt\box0\relax}
\def\unionstick{\mathop{\copy\noforkbox}\limits}

\def\nonfork_#1{\unionstick_{\textstyle #1}}

\setbox0\hbox{$\textstyle\smile$}
\setbox1\hbox to \wd0{\hfil{\sl /\/}\hfil}
\setbox2\hbox to \wd0{\hfil\vrule height 10pt depth -2pt width
              \forklinewidth\hfil}
\newbox\doesforkbox
\setbox\doesforkbox\hbox{\lower 2pt\box1 \lower 2pt\box2\lower2pt\box0\relax}
\def\nunionstick{\mathop{\copy\doesforkbox}\limits}

\def\fork_#1{\nunionstick_{\textstyle #1}}

\newcommand{\dnf}{\unionstick}
\newcommand{\s}{\mathfrak{s}}
\newcommand{\nf}{\unionstick}
\newcommand{\dnfb}[4]{#2 \overset{#4}{\underset{#1}{\overline{\nf}}} #3}

\usepackage{geometry}
\usepackage{amsmath}
\usepackage{amssymb}
\usepackage{cite}
\usepackage{amsthm}  
\usepackage{tikz-cd}
\usetikzlibrary{er,positioning}

\newtheorem{same}{This should never appear}[section]
\newtheorem{defin}[same]{Definition}

\newtheorem{remark}[same]{Remark}
\newtheorem{theorem}[same]{Theorem}
\newtheorem{example}[same]{Example}
\newtheorem{lemma}[same]{Lemma}
\newtheorem{fact}[same]{Fact}
\newtheorem{question}[same]{Question}
\newtheorem{cor}[same]{Corollary}
\newtheorem{prop}[same]{Proposition}

\newtheorem{hypothesis}[same]{Hypothesis}
\newtheorem{nota}[same]{Notation}

\newtheorem*{theorem1}{Theorem 4.8}
\newtheorem*{theorem2}{Theorem 5.11}

\newtheorem{defin*}{Definition}
\newtheorem*{theorem*}{Theorem}

\makeatletter
\newcommand{\skipitems}[1]{%
  \addtocounter{\@enumctr}{#1}%
}

\newcommand{\rest}{\mathord{\upharpoonright}}
\newcommand{\id}{\textrm{id}}

\newcommand{\Kk}{\mathbf{K}}

\newcommand{\M}{\mathcal{M}}
\newcommand{\K}{\mathcal{K}}

\newcommand{\km}{\mathcal{K}_\mathcal{M}}

\newcommand{\LS}{\operatorname{LS}}

\newcommand{\leap}[1]{\le_{#1}}

\newcommand{\lea}{\leap{\Kk}}

\newcommand\cm{\mathcal {M}}

\title{Relative injective modules, superstability and noetherian categories}

\author{Marcos Mazari-Armida}

\thanks{The first author's research was partially supported by an AMS-Simons Travel Grant 2022--2024 and by an NSF grant DMS-2348881.}

\address{\newline Marcos Mazari-Armida \newline Department of Mathematics \newline Baylor University \newline Waco, Texas, USA}
\urladdr{https://sites.baylor.edu/marcos\_mazari/}
\email{marcos\_mazari@baylor.edu}

\author[ Ji\v{r}\'{\i} Rosick\'{y}]
{ Ji\v{r}\'{\i} Rosick\'{y}}
\thanks{The second author's research was supported by the Grant Agency of the Czech Republic under the grant 22-02964S} 
\address{
\newline  Ji\v{r}\'{\i} Rosick\'{y}\newline
Department of Mathematics and Statistics,\newline
Masaryk University, Faculty of Sciences,\newline
Kotl\'{a}\v{r}sk\'{a} 2, 611 37 Brno,\newline
Czech Republic}
\email{rosicky@math.muni.cz}

\begin{document}

\begin{abstract} 
We study classes of modules closed under direct sums, $\M$-submodules and $\M$-epimorphic images where $\M$ is either the class of embeddings, $RD$-embeddings or pure embeddings.

 We show that the $\M$-injective modules of theses classes satisfy a Baer-like criterion. 
In particular, injective modules, $RD$-injective modules, pure injective modules, flat cotorsion modules and $\s$-torsion pure injective modules satisfy this criterion. 
The argument presented is a model theoretic one. We use in an essential way stable independence relations which generalize Shelah's non-forking  to abstract elementary classes.

We show that the classical model theoretic notion of superstability is equivalent to the algebraic notion of a noetherian category for these classes. We use this equivalence to characterize noetherian rings, pure semisimple rings, perfect rings and finite products of finite rings and artinian valuation rings via superstability.

\end{abstract}


\maketitle

{\let\thefootnote\relax\footnote{{AMS 2020 Subject Classification:
Primary:  03C48, 13L05. Secondary: 03C45, 03C60, 16D50, 16B70, 16D10, 18G05.

}}}

\tableofcontents

\section{Introduction}

Injective modules and their generalizations play a key role in homological algebra, commutative algebra, ring theory, module theory and the model theory of modules. 
Recall that a module $E$ is injective if for every $f: A \to B$ an embedding and $g: A \to E$ a homomorphism, there is $h: B \to E$ a homomorphism such that $h \circ f = g$. 

A fundamental result regarding injective modules is \emph{Baer's Criterion} \cite{baer} which states that to test for injectivity it is enough to consider embeddings $f: I \to R$ where $I$ is a (left) ideal of the ring $R$ instead of considering \emph{all} embeddings $f: A \to B$.

The main result of the first half of the paper is that there is a Baer-like criterion for relative injective modules.   More precisely, we will study relative injective modules for pairs  $(K, \M)$ such that $K$ is a class of $R$-modules for a fixed ring $R$, $\M$ is either the class of embeddings, $RD$-embeddings or pure embeddings and $K$ is closed under: direct sums, $\M$-submodules and $\M$-epimorphic images. For the rest of the introduction we will refer to this framework as the \emph{main assumption}.

We say that  $E$ is a $\M$-injective module of $K$ if and only if $E \in K$ and for every  $f: A \to B$ a $\M$-embedding with $A, B \in K$ and $g: A \to E$ a homomorphism there is $h : B \to E$ a homomorphism such that $h \circ f = g$. We obtain the following Baer-like criterion  for pairs  $(K, \M)$ satisfying the main assumption. 

\begin{theorem1}  $E$ is a $\M$-injective module of $K$ if and only if for every  $f: A \to B$ a  $\M$-embedding with $A, B \in K$ and $\|A\|, \| B \| \leq  \text{card}(R) + \aleph_0$, and $g: A \to E$ a $R$-homomorphism,  there is $h : B \to E$ a $R$-homomorphism such that $h \circ f = g$. 
\end{theorem1}

The framework of the theorem is very general (see Example \ref{ex1}). In particular it applies to  injective modules, $RD$-injective modules, pure injective modules, flat cotorsion modules and $\s$-torsion pure injective modules. Theorem \ref{s-inj} had already been obtained for injective modules \cite{baer}, pure injective modules \cite[1.3]{satr} and flat cotorsion modules \cite[1.1]{satr}, but it is new for $RD$-injective modules and $\s$-torsion pure injective modules. Moreover, the methods used to obtain the previously known results are different for each class of modules.  For instance, the standard proof for flat cotorsion modules uses that the class of flat modules is $(\text{card}(R) + \aleph_0)^+$-deconstructible which is a deep result in module theory (see for example \cite[6.17]{gotr}).


The argument presented to show Theorem \ref{s-inj} is a model theoretic one. We use in an essential way stable independence relations which generalize the central model theoretic notion of Shelah's non-forking which in turn generalizes linear and algebraic independence. More precisely,  we show that if a pair $(K, \M)$ satisfies the main assumption, then the abstract elementary class obtained by  taking $K$ with $\M$-embeddings  has a stable independence relation with the $(<\aleph_0)$-witness property (see Theorem \ref{s-ind}).  We use the stable independence relation to decompose any $\M$-embeddings into smaller \emph{independent pieces} (see Lemma \ref{deco}). Then we use the independence of the smaller pieces together with the closure of the class under pushouts to extend the original embedding. This last argument is similar to \cite[3.1]{lrvcell}.

The second part of the paper characterizes the classical model theoretic notion of superstability via noetherian categories.  An abstract elementary class is \emph{superstable} if it has a unique limit model up to isomorphisms in a tail of cardinals. Intuitively a \emph{limit model} is a universal model with some level of homogeneity.  We say that a pair $(K, \M) $ is \emph{noetherian} if every direct sum of $\M$-injective modules of $K$  is a $\M$-injective module of $K$. This generalizes the notion of a noetherian ring by a classical result of Bass-Papp (see for example \cite[3.46]{lam2}) and was already considered in the seventies \cite[p. 123]{ste}.

The main result of the second part of the paper is that the following result holds for pairs  $(K, \M)$ satisfying the main assumption. 

\begin{theorem2} The following are equivalent.
\begin{enumerate}
\item $(K, \M)$ is noetherian.
\item The abstract elementary class obtained by  taking $K$ with $\M$-embeddings is superstable.
\end{enumerate}
\end{theorem2} 

Similar results have been obtained for certain classes of modules with embeddings \cite[\S 3]{maz1} and with pure embeddings \cite[\S 4]{maz1}, \cite[\S 3]{m3} \cite[\S 3]{maz2}, \cite[\S 4]{m-tor}, but the result is new in this generality even when $\M$ is the class of pure embeddings and for the specific case of $R$-modules with $RD$-embeddings.

The proof of the previous theorem is similar to that of \cite[\S 3]{m3} and \cite[\S 3]{maz2}, but there were two crucial difficulties that we needed to overcome to generalize the argument. The first was to show that \emph{long} limit models are relative injective modules without a syntactic characterization of the relative injective modules (see Lemma \ref{sat} and Proposition \ref{lim-inj}). We were able to overcome this obstacle using the Baer-like criterion of the first part of the paper. The second was to show that relative injective modules satisfy a Schr\"{o}der-Bernstein type property without the existence of relative injective envelopes (see Lemma \ref{csb}) . We were able to overcome this obstacle by assuming that $(K, \M)$ is noetherian. It is open whether the result is true without this assumption (see Question \ref{q3}).

We use Theorem \ref{equiv} to characterize classical classes of rings such as noetherian rings, pure semisimple rings, perfect rings and finite products of finite rings and artinian valuation rings. Most of these characterizations had already been obtained in \cite{maz1}, \cite{m3}, \cite{maz2} but the case of finite products of finite rings and artinian valuation rings is new (see Corollary \ref{rd-ring}).

Another result of the paper is that $RD$-embeddings are cofibrantly generated in the class of $R$-modules, i.e., they are generated from a set of morphisms by pushouts, transfinite composition and retracts (see Corollary \ref{rd-co}). The result follows from the existence of a stable independence relation and \cite[3.1]{lrvcell}. The result for pure embeddings was just recently obtained in \cite{lprv}.

The paper is organized as follows. Section 2 presents the main framework of the paper together with some basic results and background material on abstract elementary classes. Section 3  shows the existence of a stable independence relation for classes satisfying the main assumption of the paper. Section 4 uses the stable independence relation of Section 3 to show a Baer-like criterion for relative injective modules. Section 5 presents the equivalence between superstable AECs and noetherian categories for classes satisfying our main assumption and uses this equivalence to characterize some classical rings. Furthermore in Section 5, we provide a couple of extensions to the equivalence between superstable AECs and noetherian categories. 

We would like to thank Ivo Herzog for pointing out Example \ref{ex1}.(6) and Wentao Yang for comments that helped improved the paper. We are grateful to the referee for many comments that significantly improved the presentation of the paper.

\section{Basic results and preliminaries}

We present the main framework of the paper. We recall some notions concerning abstract elementary classes that are used in this paper. The proof of the Baer-like criterion for relative injective modules only uses  basic notions (up to Proposition \ref{ap}). 

\subsection{Main framework} All rings considered in this paper are associative with unity and all modules are left $R$-modules unless specified otherwise. We denote the class of all embeddings by $Emb$ and we write $A \leq_{Emb} B$ if $A$ is a submodule of $B$. 

An exact sequence of modules 

\[
\xymatrix{
  0 \ar[r] & A \ar[r]^-{f} & B \ar[r]^-{g} & C \ar[r] & 0
}
\]
is \emph{pure-exact} if  every system of linear equations with parameters in $f[A]$ which has a solution in $B$ has a solution in $f[A]$. In this case we say that $f$ is a \emph{pure embedding} and $g$ is a \emph{pure epimorphism}. We will denote the class of all pure embeddings by \textit{Pure}.  If $f$ is the inclusion, we say that $A$ is a \emph{pure submodule} of $B$ and denote it by $A \leq_{Pure} B$.

An exact  sequence of modules is \emph{RD-exact} if $f[A] \cap rB = rf[A]$ for every $r \in R$. In that case we say that $f$ is a \emph{$RD$-embedding} and $g$ is a \emph{$RD$-epimorphism}. We will denote the class of $RD$-embeddings by \textit{RD}.  If $f$ is the inclusion, we say that $A$ is a \emph{$RD$-submodule} of $B$ and denote it by $A \leq_{RD} B$.

 It is clear that pure embeddings are $RD$-embedding, but the other inclusion might fail. The rings where both notions coincide are called $RD$-rings \cite{ppr}.
 
 We introduce the framework of the paper.

\begin{hypothesis}\label{hyp2} Let $(K, \M)$ be a pair such that:
\begin{enumerate}
\item $K$ is a class of $R$-modules for a fixed ring $R$, 
\item $\M$ is either the class of embeddings, $RD$-embeddings or pure embeddings, and 
\item $K$ is closed under:
\begin{enumerate}
\item direct sums, 
\item $\M$-submodules, i.e., if $A \in K$ and $B \leq_\M A$, then $B \in K$, and 
\item $\M$-epimorphic images, i.e., if $f: A \to B$ is a $\M$-epimorphism and $A \in K$, then $B \in K$. 
\end{enumerate} 
\end{enumerate} 
\end{hypothesis}

\begin{remark}
Most of the main results of the paper assume the above hypothesis, but we will explicitly mention when it is assumed.
\end{remark}

It is clear that if $(K, Emb)$ satisfies Hypothesis \ref{hyp2}, then $(K, RD)$ satisfies Hypothesis \ref{hyp2}, and if  $(K, RD)$ satisfies Hypothesis \ref{hyp2}, then $(K, Pure)$ satisfies Hypothesis \ref{hyp2}. The other inclusions might fail. An easy example is that torsion-free abelian groups with $RD$-embeddings (pure embeddings) satisfy  Hypothesis \ref{hyp2}, but  torsion-free abelian groups with embeddings do not satisfy it.

We give some examples of classes satisfying Hypothesis \ref{hyp2}.

\begin{example}\label{ex1}\
\begin{enumerate}
\item The class of all modules with embeddings, $RD$-embeddings or pure embeddings. 
\item The class of torsion abelian groups with embeddings, $RD$-embeddings or pure embeddings. 
\item An $F$-class (in the sense of \cite[\S 2]{hero}) with pure embeddings. Some examples of $F$-classes are: the class of all $R$-modules,  flat $R$-modules and $\s$-torsion modules. Additional examples are given in \cite[\S 2]{hero}. 
\item $Ker(Tor_1(\mathcal{B}, -))$ with pure embeddings for  $\mathcal{B}$ a class of right $R$-modules \cite[4.3]{hojo}.
\item Let $R$ be a left coherent ring.  $Ker(Ext^1(\mathcal{B}, -))$  with pure embeddings  for  $\mathcal{B}$ a class of finitely presented left $R$-modules \cite[4.3]{hojo}. 
\item  $\sigma[A]$  with embeddings where $A$ is a module and $\sigma[A]$ is the full subcategory of the category of modules whose objects are all modules subgenerated by $A$ (see \cite[\S 15]{wis}).  
\end{enumerate}
\end{example}

Throughout the paper we will use the following basic notions from the  model theory of modules.  A formula $\varphi(\bar{x})$ is a positive primitive formula ($pp$-formula for short) provided it is equivalent,  relative to the theory of (left) $R$-modules, to a formula of the form:

 \[ \exists w_1,  \ldots, \exists w_l \; \bigwedge_{j=1}^{m} \; \left(
\sum_{i=1}^{l} r_{j,i} w_i + \sum_{k=1}^{n} s_{j,k}x_{k} = 0 \right)\] for  $r_{j,i}, s_{j,k} \in R$ for every $i \in \{ 1,..., l\}, j \in \{1,...,m\}, k\in \{1,...,n\}$. It is known, see for example \cite[2.1.6]{prest09}, that $A \leq_{Pure} B$ if and only if $A$ is a submodule of $B$ and for every $a \in A$ and $\varphi(x)$ a $pp$-formula, if $B \vDash \varphi(a)$, then $A \vDash \varphi(a)$. 

A result we will often use is that if $f: A \to B$ is a $R$-homomorphism, $a \in A$ and $\varphi(x)$ is a $pp$-formula such that $A \vDash \varphi(a)$, then $B \vDash \varphi(f(a))$ (see for example \cite[2.7]{prest}). 

 A formula $\varphi(x)$ is a $RD$-formula if it is of the form $\exists w ( rw =x)$ for $r \in R$. It is clear that $A \leq_{RD} B$ if and only if $A$ is a submodule of $B$ and for every $a \in A$ and $\varphi(x)$ a $RD$-formula, if $B \vDash \varphi(a)$, then $A \models \varphi(a)$. 
 
 Recall that an embedding  $f: A \to B$ \emph{splits} if there is $\pi: B \to A$ such that $\pi \circ f = \id_A$. In our setting,  an embedding  $f: A \to B$ splits if and only if  $f[A]$ is a direct summand of $B$, i.e., there is a module $C$ such that $B =  f[A] \oplus C$.

\begin{prop}\label{v-basic}  Assume $(K, \M)$ satisfy Hypothesis \ref{hyp2}. 
\begin{enumerate}
\item  $K$ is closed under direct summands.
\item Split embeddings are contained in $\M$, i.e.,  if $A \in K$ is a direct summand of $B\in K$, then $A \leq_\M B$.
\item 
If $A \leq_\M C$ and $B \leq_\M D$, then $A \oplus B \leq_\M C \oplus D$.
\end{enumerate}
\end{prop} 
\begin{proof}
(1) and (2) are clear, so we only show (3). We show first the case when $\M= RD$. It is clear that $A \oplus B$ is a submodule of   $C \oplus D$. Let $r \in R$, $(a, b)  \in A \oplus B$ and $(c, d) \in C \oplus D$ such that $(a, b) = r(c,d)$. Projecting onto $C$  we get that $a = rc$ and since $A \leq_{RD} C$ there is $a' \in A$ such that $a = ra'$. Similarly there is $b' \in B$ such that $b = rb'$. Hence $(a, b) = r(a', b') \in r(A \oplus B)$. Therefore, $A \oplus B \leq_{RD} C \oplus D$.

We turn to the case when $\M=Pure$. The proof is similar to the one above but we provide some of the details. Let $\varphi(x)$ be a $pp$-formula and $(a, b)  \in A \oplus B$ such that $C \oplus D \vDash \varphi((a,b))$. Projecting onto $C$, we get that $C \vDash \varphi(a)$. As  $A \leq_{Pure} C$, $A \vDash \varphi(a)$. Hence $A \oplus B \vDash \varphi( (a, 0))$. Similarly, one can show that $A \oplus B \vDash \varphi( (0, b))$. Since $\varphi(A \oplus B )$ is a subgroup of $A \oplus B$ (see for example \cite[2.2]{prest}), $A \oplus B \vDash \varphi((a,b))$. Therefore, $A \oplus B \leq_{Pure} C \oplus D$. \end{proof}

\begin{remark} Recall that the pushout of a pair of morphisms $(f_1: M \to N_1, f_2: M \to N_2)$ in the category of $R$-modules with $R$-homomorphisms is given by: 

 $$
  \xymatrix{
    N_1 \ar[r]^{\bar{f_2}} & P  \\
    M  \ar [u]^{f_1} \ar[r]_{f_2}  & N_2\ar[u]_{\bar{f_1}}
  }
$$where  $P= (N_1 \oplus N_2)/ \{ (f_1(m), -f_2(m)) : m \in M\} $, $ \bar{f_1} : n_2 \mapsto [(0, n_2)]_P$ and $\bar{f_2}: n_1 \mapsto [(n_1, 0)]_P$.

Moreover,  for every $(Q, h_1, h_2)$ such that $h_1 \circ f_1 = h_2 \circ f_2$, the unique $t: P \to Q$ such that $t \circ \bar{f_2} = h_1$ and $t \circ \bar{f_1} = h_2$ is given by $t([(n_1, n_2)]_P)= h_1(n_1) + h_2(n_2)$.   


\end{remark}

\begin{lemma}\label{push}
Assume $(K, \M)$ satisfy Hypothesis \ref{hyp2}. 
If $M,N_1, N_2 \in K$, $f_1: M \to N_1$ is a $\M$-embedding, $f_2: M \to N_2$ is a $R$-homomorphism and $(P, \bar{f_1}: N_2 \to P, \bar{f_2}: N_1 \to P)$ is the pushout of $(f_1, f_2)$ in the category of $R$-modules with $R$-homomorphisms, then $P \in K$ and $\bar{f_1}$ is a $\M$-embedding. Moreover if $f_2$ is also a $\M$-embedding, then  $\bar{f_2}$ is a $\M$-embedding.
\end{lemma}
\begin{proof}  The key to show that $P \in K$ is to show that the canonical epimorphism  $\pi: N_1 \oplus N_2 \to (N_1 \oplus N_2)/ \{ (f_1(m), -f_2(m)) : m \in M\}$ is a $\M$-epimorphism. We show the case when $\M = Pure$; the case when $\M = RD$ is similar. We show that  $\{ (f_1(m), -f_2(m)) : m \in M\} \leq_{Pure} N_1 \oplus N_2$. Let $\varphi(x)$ be a $pp$-formula such that $N_1 \oplus N_2 \vDash  \varphi(  (f_1(m_0), -f_2(m_0))$ for $m_0 \in M$. Projecting onto $N_1$ we get that $N_1 \vDash \varphi(f_1(m_0))$. Since $f_1: M \to N_1$ is a pure embedding, $M \vDash \varphi(m_0)$. Since $f_1, f_2$ are $R$-homomorphisms, $h: M \to  \{ (f_1(m), -f_2(m)) : m \in M\}$ given by $h(m)= (f_1(m), -f_2(m))$ is a $R$-homomorphism. Hence $\{ (f_1(m), -f_2(m)) : m \in M\} \vDash \varphi(  (f_1(m_0), -f_2(m_0))$. Therefore, $\pi$ is a pure epimorphism.

It is well-known that $\bar{f_1} \in \M$ for $\M = Emb$ and $\M =Pure$ (see for example \cite[2.1.13]{prest09}); so we only present the proof for $\M = RD$. Let  $(P= (N_1 \oplus N_2)/ \{ (f_1(m), -f_2(m)) : m \in M\} ,\; \bar{f_1} : n_2 \mapsto [(0, n_2)]_P  , \; \bar{f_2}: n_1 \mapsto [(n_1, 0)]_P )$ be the pushout of $(f_1, f_2)$ such that $f_1$ is a $RD$-embedding. We show that $\bar{f_1}$ is a $RD$-embedding.

 Let $r \in R$, $[(0, n_2)]_P \in \bar{f_1}[N_2]$ and $[(\ell_1, \ell_2)]_P \in P$ such that $[(0, n_2)]_P = r[(\ell_1, \ell_2)]_P$. Then there is an $m \in M$ such that $r \ell_1 = f_1(m)$ and $r\ell_2  - n_2 = -f_2(m)$. Since $r \ell_1 = f_1(m)$ and $f_1: M \to N_1$ is a $RD$-embedding there is $m' \in M$ such that $r m' = m$. It is easy to show that $r[(0, \ell_2 + f_2(m'))]_P = [(0, n_2)]_P$. Hence $[(0, n_2)]_P \in r \bar{f_1}[N_2]$. Therefore, $\bar{f_1}$ is a $RD$-embedding. The proof that  $\bar{f_2}$ is a $RD$-embedding if $f_2$ is a $RD$-embedding is similar. \end{proof}

\subsection{Abstract elementary classes} Abstract elementary classes (AECs for short) were introduced by Shelah \cite{sh88}. An AEC is a pair $\Kk=(K \lea)$ where $K$ is a class of structures in a fixed language\footnote{In this paper, the language will always be $\{0, +,-\} \cup \{ r\cdot   : r \in R \}$ where $R$ is a fixed ring and $r \cdot$ is interpreted as multiplication by $r$ for each $r \in R$.} and $\lea$ is a partial order on $K$ extending the substructure relation such that $\Kk$ is closed under isomorphisms and satisfies the  coherence property, the L\"{o}wenheim-Skolem-Tarski axiom and the Tarski-Vaught axioms. The L\"{o}wenheim-Skolem-Tarski axiom is an an instance of the Downward L\"{o}wenheim-Skolem theorem and the Tarski-Vaught axioms assure us that the class is closed under directed colimits. The reader can consult the definition in \cite[4.1]{baldwinbook09}.

\begin{lemma}\label{aec} 
If $(K, \M)$ satisfies Hypothesis \ref{hyp2}, then $\km:=(K, \leq_\M)$ is an abstract elementary class with $\LS(\km) = \text{card}(R)  +\aleph_0$. 
\end{lemma} 
\begin{proof} 
The only two axioms that require an argument are the L\"{o}wenheim-Skolem-Tarski axiom and the Tarski-Vaught axioms. The fact that $\LS(\km) = \text{card}(R)  +\aleph_0$ follows from the assumption that $K$ is closed under $\M$-submodules and that every set $A$  can be extended to a pure submodule of cardinality bounded by $\text{card}(R) + \aleph_0 + |A|$ (see for example \cite[2.1.21]{prest09}).  Hence, the L\"{o}wenheim-Skolem-Tarski axiom holds.

  We show that the Tarski-Vaught axioms hold. Suppose $\delta$ is a limit ordinal and $\{ M_i \in K : i < \delta \}$ is an increasing chain. Then 
$\bigoplus_{i < \delta} M_i \in K$ as $K$ is closed under direct sums. Let $f: \bigoplus_{i < \delta} M_i \to \bigcup_{i < \delta} M_i$ be given by $(m_i)_{i <\delta} \mapsto \sum_{i< \delta} m_i$. We show $f$ is a pure epimorphism (even if $\M = Emb$ or $RD$).  We show that if $\varphi(x)$ is a $pp$-formula and $M_\delta =\bigcup_{i < \delta} M_i \vDash \varphi(m)$ for $m \in M_\delta$, then there is $n \in \bigoplus_{i < \delta} M_i
$ such that $f(n) = m$ and $\bigoplus_{i < \delta} M_i \vDash \varphi(n)$. This is enough by \cite[2.1.14]{prest09}.

Assume $\varphi(x) =  \exists \bar{w} \theta(\bar{w}, x)$ for $\theta(\bar{w}, x)$ a quantifier free formula and $M_\delta \vDash \varphi(m)$ for $m \in M_\delta$. Then there is $\bar{m}^* \in  M_\delta$ such that $M_\delta \vDash \theta(\bar{m}^*, m)$. As $ M_\delta =\bigcup_{i < \delta} M_i $, there is  $j < \delta$, such that $(\bar{m}^*,m) \in M_j$. Since $\theta(\bar{w}, x)$ is  a quantifier free formula, $M_j  \vDash \theta(\bar{m}^*, m)$. Hence $M_j \vDash \varphi(m)$. Let $i_j: M_j \hookrightarrow \bigoplus_{i < \delta} M_i$ be the canonical injection.  Therefore, $\bigoplus_{i < \delta} M_i \vDash \varphi(i_j(m))$ and $f(i_j(m))=m$. Hence $f$ is a pure epimorphism.

Since $f$ is a pure epimorphism,  it is a $\M$-epimorphism. As $K$ is closed under $\M$-epimorphic images,  $ \bigcup_{i < \delta} M_i \in K$. It is clear that if $M_i \leq_\M N$ for some $N \in K$, then  $ \bigcup_{i < \delta} M_i \leq_\M N$. Hence the Tarski-Vaught axioms hold. \end{proof}

\begin{remark} The previous proof shows that if $(K, Pure)$ satisfies Hypothesis \ref{hyp2}, then not only $(K, \leq_{Pure})$ is an abstract elementary class but also $(K, \leq_{Emb})$ and $(K, \leq_{RD})$. 

\end{remark}

Given $M \in K$, $|M|$ is the underlying set of $M$ and $\| M\|$ is the cardinality of $M$. $f: M \to N$ is a $\Kk$-embedding if  $f: M \cong f[M]$ and $f[M] \lea N$, in particular we always assume that $M,N \in K$. Observe that $\Kk$-embeddings are injective functions.

 An AEC has the \emph{amalgamation property} if any span $M \lea N_1, N_2$ can be completed to a commutative square; the \emph{joint embedding property} if any two models can be embedded into a third model;  and \emph{no maximal models} if every model can be properly extended.  

\begin{prop}\label{ap}
If $(K, \M)$ satisfies Hypothesis \ref{hyp2}, then $\km$ has the amalgamation property, the joint embedding property and no maximal models. 
\end{prop}
\begin{proof}
The result follows from the closure of $K$ under direct sums and Lemma \ref{push}.
\end{proof}

\begin{remark} 
If $(K, \M)$ satisfies Hypothesis \ref{hyp2}, $\km$ actually has the disjoint amalgamation property\footnote{ An AEC has the \emph{disjoint amalgamation property} if for any span $M \lea N_1, N_2$ with $N_1 \cap N_2 = M$ there are $N$, $f_1: N_1 \to N$ and $f_2: N_2 \to N$ such that $f_1\rest_M= f_2\rest_M$ and $f_1[N_1] \cap f_2[N_2] = f_1[M]$.} because $K$ is closed under pushouts of $\M$-embeddings (see Lemma \ref{push}). Since we will not use this result in what follows, we do not give any details. 
\end{remark} 

\begin{nota}\
\begin{itemize}
\item  If $\lambda$ is a cardinal and $\Kk$ is an AEC, then $\Kk_{\lambda}=\{ M \in K : \| M \|=\lambda \}$ and $\Kk_{\leq \lambda}=\{ M \in K : \| M \| \leq \lambda \}$
\item We write $f: M \xrightarrow[A]{} N$ if $f$ is an embedding and $f\rest_A = \id_A$. 
\end{itemize}

\end{nota}

\begin{defin}
$M$ is $\lambda$-\emph{universal over} $N$ if and only if $N \leq_\Kk M$
 and for any $N^* \in \Kk_{\leq\lambda}$ such that
$N \leq_\Kk N^*$, there is $f: N^* \xrightarrow[N]{} M$ a $\Kk$-embedding. $M$ is \emph{universal
over} $N$ if and only if $\| N\|= \| M\| $ and $M$ is $\| M
\|$-\emph{universal over} $N$. 
\end{defin}

\begin{defin} Assume $\lambda > \LS(\Kk)$. $M$ is $\lambda$-saturated if for every $N \lea M$ with $\| N \| < \lambda$,  $M$ is $\mu$-universal over $N$ for every $\mu < \lambda$. 

\end{defin}

The above definition is not the standard definition of saturated model for AECs, but it is equivalent if the AEC has the amalgamation property, the joint embedding property and no maximal models \cite[\S II.1.4]{shelahaecbook}.

We introduce limit models (also called brimmed models by Shelah), they were originally introduced in \cite{kosh}.

\begin{defin}\label{limit}
Let $\lambda$ be an infinite cardinal and $\alpha < \lambda^+$ be a limit ordinal.  $M$ is a \emph{$(\lambda,
\alpha)$-limit model over} $N$ if and only if there is $\{ M_i : i <
\alpha\}\subseteq \Kk_\lambda$ an increasing continuous chain such
that:
\begin{itemize}
\item $M_0 =N$ and  $M= \bigcup_{i < \alpha} M_i$, and
\item $M_{i+1}$ is universal over $M_i$ for each $i <
\alpha$.
\end{itemize}

$M$ is a $(\lambda, \alpha)$-limit model if there is $N \in
\Kk_\lambda$ such that $M$ is a $(\lambda, \alpha)$-limit model over
$N$. $M$ is a $\lambda$-limit model if there is a limit ordinal
$\alpha < \lambda^+$ such that $M$  is a $(\lambda,
\alpha)$-limit model.  $M$ is a limit model if there is an infinite cardinal $\lambda$ such that $M$ is a $\lambda$-limit model.

\end{defin}

\begin{remark}\label{uni-lim}
Observe that if $M$ is a $\lambda$-limit model, then $M$ has cardinality $\lambda$. Moreover, if $M$ is a $\lambda$-limit model, then $M$ is a universal model in
$\Kk_\lambda$ \cite[2.10]{maz}, i.e., given any $N \in \Kk_\lambda$, there is $f: N \to M$ a $\Kk$-embedding.
\end{remark}

\begin{fact}\label{sat-lim} 
Let $\Kk$ be an AEC with amalgamation, joint embedding and no maximal models. If $\lambda > \LS(\Kk)$ and $M$ is a $(\lambda, \alpha)$-limit model for $\alpha \in[ \LS(\Kk)^+, \lambda]$ a regular cardinal, then $M$ is an $\alpha$-saturated model.
\end{fact}

Using limit models we introduce a non-standard definition of stability. The definition given is equivalent to the standard definition of stability (there are few Galois types over models) if the AEC has the amalgamation property, the joint embedding property and no maximal models \cite[\S II]{shelahaecbook}, \cite[2.9]{tamenessone}. We use this as our definition of stability because all the AECs considered in this paper have these three properties.

\begin{defin}  $\Kk$ is $\lambda$-stable if $\Kk$ has a $\lambda$-limit model. $\Kk$ is stable if there is a $\lambda$ such that $\Kk$ is $\lambda$-stable.
\end{defin}

Using limit models we introduce a definition of superstability. The definition given is equivalent to any definition of superstability used in the context of AECs  if the AEC has the amalgamation property, the joint embedding property, no maximal models, and is tame\footnote{An AEC is tame if two distinct Galois types can be distinguished by a \emph{small} set.}  \cite[1.3]{grva}, \cite{vaseyt}. We use this as our definition of superstability as all the AECs considered in this paper have these properties.

\begin{defin}
$\Kk$ is \emph{superstable} if and only if $\Kk$ has uniqueness of limit models in a tail of cardinals, i.e., there is a cardinal $\mu \geq \LS(\Kk)$ such that there is a unique $\lambda$-limit model for every $\lambda \geq \mu$.
\end{defin}

This is all the theory of abstract elementary classes used on this paper. More detailed introductions to the subject are presented in \cite{baldwinbook09}, \cite{shelahaecbook} and \cite{grossberg2002}.

\section{A stable independence relation}

 We begin by recalling the notion of a stable independence relation. We show that  classes satisfying Hyothesis \ref{hyp2} have a stable independence relation which will be the key to prove a Baer-like criterion for relative injective modules.


\subsection{Preliminaries} Independence relations on arbitrary categories were introduced and studied in detail in \cite{lrv1}, these extend Shelah's notion of non-forking which in turn extends linear and algebraic independence. In this subsection, we discuss various properties of an abstract independence relation.

\begin{defin}[{ \cite[3.4]{lrv1}}] An independence relation on a category $\mathcal{K}$ is a set $\dnf$ of commutative squares such that for any commutative diagram:

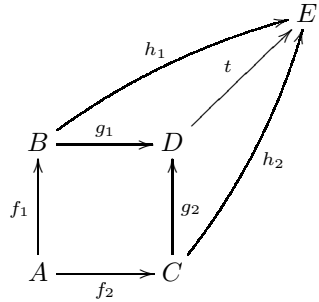
\[
  \xymatrix@=3pc{
    & & E \\
    B \ar[r]^{g_1}\ar@/^/[rru]^{h_1} & D \ar[ru]^t  & \\
    A \ar [u]^{f_1} \ar[r]_{f_2} & C \ar[u]_{g_2} \ar@/_/[ruu]_{h_2} &
  }
\]

we have that $(f_1, f_2, g_1, g_2) \in \dnf$ if and only if $(f_1, f_2, h_1, h_2) \in \dnf$.

\end{defin} 

An independence relation $\dnf$ is \emph{weakly stable} if it satisfies: symmetry \cite[3.9]{lrv1},  existence \cite[3.10]{lrv1}, uniqueness \cite[3.13]{lrv1}, and transitivity \cite[3.15]{lrv1}.

In \cite[3.24]{lrv1} the notion of a stable independence relation is introduced for any category $\mathcal{K}$. As we will only study independence relations on AECs, we restrict our discussion to stable independence relations on AECs.

\begin{nota} Let $\Kk$ be an AEC with an independence relation $\dnf$.

We write $M_1 \dnf^{M_3}_{M_0} M_2$ if $M_0 \lea M_1, M_2 \lea M_3$ and $(i_{0,1}, i_{0,2},i_{1,3},i_{2,3}) \in \dnf$ where $i_{\ell, m}$ is the inclusion map for every $\ell, m$. \end{nota}

On AECs, the independence notion can be extended to sets.

\begin{defin} Let $\Kk$ be an AEC with an independence relation $\dnf$. 
$A$ is (bar-)free from $B$ over $N_0$ in $N_3$, denoted by $\dnfb{N_0}{A}{B}{N_3}$, if $N_0 \lea N_3$, $A \cup B \subseteq |N_3|$ and there are $M_1, M_2, M_3 \in \Kk$ such that $A \subseteq |M_1|$, $B \subseteq |M_2|$, $N_3 \lea M_3$ and $M_1 \dnf^{M_3}_{N_0} M_2$. 
\end{defin}

The properties of $\dnfb{}{}{}{}$ we will use in this paper are listed in the next fact.

\begin{fact}[{\cite[8.4]{lrv1}}]\label{b-dnf}
Let $\Kk$ be an AEC and assume $\dnf$ is a weakly stable independence relation on $\Kk$. 
\begin{enumerate}
\item Assume  $M_0 \lea M_1, M_2 \lea M_3$. $M_1 \dnf_{M_0}^{M_3} M_2$ if and only if $ \dnfb{M_0}{M_1}{M_2}{M_3} $.
\item (Monotonicity) If $\dnfb{M_0}{A}{B}{M_3}$, $A_0 \subseteq A$, and $B_0 \subseteq B$, then $ \dnfb{M_0}{A_0}{B_0}{M_3}$.
\item (Base monotonicity) If  $\dnfb{M_0}{A}{B}{M_3}$, $M _0 \lea M_2 \lea M_3$ and $|M_2| \subseteq B$, then  $\dnfb{M_2}{A}{B}{M_3}$.
\item (Symmetry)  $\dnfb{M_0}{A}{B}{M_3}$ if and only if  $ \dnfb{M_0}{B}{A}{M_3}$.
\end{enumerate}
\end{fact}

\begin{defin}[{\cite[8.7]{lrv1}}] Let $\Kk$ be an AEC with an independence relation $\dnf$  and $\theta$ be a cardinal. $\dnf$ has the (right) $(< \theta)$-witness property if: $M_1 \dnf_{M_0}^{M_3} M_2$ if and only if $M_0 \lea M_1, M_2 \lea M_3$ and for every $A \subseteq |M_2|$, if $|A| < \theta$ then $\dnfb{M_0}{M_1}{A}{M_3} $. $\dnf$ has the (right) witness property if there is cardinal $\theta$ such that $\dnf$ has the (right) $(< \theta)$-witness property.
\end{defin}

We are ready to introduce stable independence relations.

\begin{defin}[{\cite[3.24, 8.14]{lrv1}}] Let $\Kk$ be an AEC. $\dnf$ is \emph{stable independence relation}  (on $\Kk$) if $\dnf$ is weakly stable and satisfies the witness property and local character \cite[8.6]{lrv1}.
\end{defin}

\subsection{Main results} The following definition appears first in \cite[2.2]{lrvcell}, but under additional assumptions on the category  $(K, {}_RMorphisms)$.

\begin{defin}\label{indp} Assume $(K, \M)$ satisfies Hypothesis \ref{hyp2}.
$(f_1, f_2, h_1, h_2) \in \dnf$ if and only if all the arrows of the outer square are $\M$-embeddings,  $(P, g_1, g_2)$ is the pushout of $(M, f_1, f_2)$, and the unique map $t: P \to Q$ is a $\M$-embedding:

 
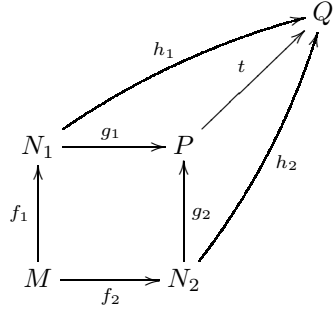
\[
  \xymatrix@=3pc{
    & & Q \\
    N_1 \ar[r]^{g_1}\ar@/^/[rru]^{h_1} & P \ar[ru]^t & \\
    M \ar [u]^{f_1} \ar[r]_{f_2} & N_2 \ar[u]_{g_2} \ar@/_/[ruu]_{h_2} &
  }
\]

\end{defin}

We will show that $\dnf$ is a stable independence relation if $(K, \M)$ satisfies Hypothesis \ref{hyp2}. That $\dnf$ is a weakly stable independence relation follows from \cite{lrvcell} and Lemma \ref{push}.

\begin{fact}[{ \cite[2.7]{lrvcell}}]\label{weak-stable}
If $(K, \M)$ satisfies Hypothesis \ref{hyp2}, then $\dnf$ is a weakly stable independence relation on $\km$.

\end{fact}

We turn to show local character. 

\begin{lemma}\label{local}
If $(K, \M)$ satisfies Hypothesis \ref{hyp2}, then $\dnf$ has local character on $\km$. More precisely, if $M_1, M_2 \leq_\M N$, then there are $M_1', M_0 \in K$ such that $M_0 \leq_\M M_1', M_2 \leq_\M N$, $M_1 \leq_\M M_1'$, $\| M_0 \| \leq \| M_1 \| + \text{card}(R) + \aleph_0$ and $M'_1 \dnf^N_{M_0}  M_2$.
\end{lemma}
\begin{proof}
The case when $\M = Pure$ was shown in \cite[4.14]{maz2}.  We show the case when $\M = RD$ as the case when $\M = Emb$ is similar, but simpler. The proof follows the structure of  the proof of \cite[4.14]{maz2}, but we provide some details as there is a key difference. 

Let $\Phi$ be the following set of $pp$-formulas:

\[\{  \sum_{k=1}^{l} s_{k}x_{k} = 0 : s_k \in R \text{ for every } k \} \cup \{ \exists z \exists w ( r w = z \wedge z  = x_1 + x_2) : r \in R \}\]

Let $M_1, M_2 \leq_{RD} N$. We build two increasing continuous chains $\{M_{0,i} : i < \omega \}$ and $\{M'_{1,i} : i < \omega \}$ such that:
\begin{enumerate}
\item $M'_{1,0} = M_1$,
\item $M_{0, i} \leq_{RD} M'_{1, i+1}, M_2 \leq_{RD} N$, 
\item $\| M_{0, i} \|, \| M'_{1,i} \| \leq \| M_1 \| + \text{card}(R) + \aleph_0$, and
\item if $\bar{a} \in M'_{1,i}$, $\varphi(\bar{x}, y) \in \Phi$ and there is $m \in M_2$ such that $N \vDash \varphi(\bar{a}, m)$, then there is $b \in M_{0, i}$ such that $N \vDash \varphi(\bar{a}, b)$. 
\end{enumerate}

The construction is standard and similar to that of \cite[4.14]{maz2} so we omit it.

 We show that this is enough.  Let $M_0= \bigcup_{ i < \omega } M_{0, i}$ and $M'_1 = \bigcup_{i < \omega } M'_{1,i}$. Observe that $\| M_0 \| \leq  \| M_1 \| + \text{card}(R) + \aleph_0$ and we show that $M'_1 \dnf^N_{M_0}  M_2$.

We show that  $t:  P = (M'_1 \oplus M_2)/ \{ (m, -m) : m \in M_0 \} \to N$ given by $t([(m, n)]_P)= m + n$ is a $RD$-embedding. 

The proof that $t$ is an embedding is the same as that of \cite[4.14]{maz2} using that $x_1-x_2 - x_3 =0 \in \Phi$. We show that $t$ is a $RD$-embedding. Let $\varphi(y) = \exists w ( rw =y)$ be such that $N \vDash  \exists w ( rw =y) (m+ n)$ with $m \in M_1'$ and $n \in M_2$. Then $N \vDash \exists z \exists w( rw = z \wedge z  = x_1 + x_2) (m, n)$. Observe that $\exists z \exists w( rw = z \wedge z  = x_1 + x_2) \in \Phi$, $m \in M_1'$ and $n \in M_2$, then there is $p \in M_0$ such that $N \vDash \exists w ( rw =y) (m + p)$ by Condition (4) of the construction. As $M_1' \leq_{RD} N$, there is $m^* \in M_1'$ such that $N \vDash r m^* = m + p$. 

Since $N \vDash \exists w ( rw =y) (m+ n)$  and $N \vDash \exists w ( rw =y) (m + p)$, $N \vDash \exists w ( rw =y) (n - p)$. As $M_2 \leq_{RD} N$, there is $n^* \in M_2$ such that $N \vDash r n^* = n-p$.

It follows that $N \vDash r(m^* + n^*) = m + n$ for $m^* \in M_1'$, $n^* \in M_2$ by adding  the last equation of the previous two paragraph. Therefore,  $t$ is a $RD$-embedding and $M'_1 \dnf^N_{M_0}  M_2$.
\end{proof}

We show the last condition for $\dnf$ to be a stable independence relation. 

\begin{lemma}\label{s-witness}
If $(K, \M)$ satisfies Hypothesis \ref{hyp2}, then $\dnf$ has the $(<\aleph_0)$-witness property on $\km$.
\end{lemma}
\begin{proof} We show the case when $\M = Pure$, the case when $\M = RD$ is similar and the case when $\M = Emb$ is covered in the proof of when $\M = Pure$.

Let $M_0 \leq_{Pure} M_1, M_2 \leq_{Pure} M_3$ and assume that for every $A \subseteq |M_2|$, if $|A| < \aleph_0$ then $ \dnfb{M_0}{M_1}{A}{M_3}$.

We show that $t:  P=M_1 \oplus M_2 /\{ (m, -m) : m \in M_0 \} \to M_3$ given by $t([(m_1,m_2)]_P)=m_1 + m_2$ is a pure embedding.


We show first that $t$ is an embedding. Assume $m_1 + m_2 = m_1^* + m_2^*$ for $m_i, m_i^* \in M_i$ for $i \in \{1, 2\}$. Let $A =\{ m_2, m_2^*\}$, then $ \dnfb{M_0}{M_1}{A}{M_3}$ so there are $L_1, L_2 \leq_{Pure} L_3$ such that $|M_1| \subseteq |L_1|$, $A \subseteq |L_2|$, $M_3 \leq_{Pure} L_3$ and $L_1 \dnf_{M_0}^{L_3} L_2$ . By definition of $\dnf$ we have that in the following pushout diagram:

\[\begin{tikzcd}
                                                   &                                                                                  &  & L_3 \\
L_1 \arrow[r, "q_1"'] \arrow[rrru, bend left] & { Q=L_1 \oplus L_2 /\{ (m, -m) : m \in M_0 \} } \arrow[rru, "s" description]  &  &   \\
M_0 \arrow[u] \arrow[r]                       & L_2 \arrow[u, "q_2"] \arrow[rruu, bend right]                                      &  &  
\end{tikzcd}\]
 $s:  Q=L_1 \oplus L_2 /\{ (m, -m) : m \in M_0 \} \to L_3$ given by $s([(m_1,m_2)]_Q)=m_1 + m_2$  is a pure embedding. As $m_1 + m_2 = m_1^* + m_2^*$, it is clear that $s([(m_1,m_2)]_Q) = s([(m_1^*,m_2^*)]_Q)$. So $[(m_1,m_2)]_Q = [(m_1^*,m_2^*)]_Q$. Since $P$ and $Q$ are obtained by taking a quotient by $\{ (m, -m) : m \in M_0 \}$, it follows that  $[(m_1,m_2)]_P = [(m_1^*,m_2^*)]_P$. Hence $t$ is an embedding.

We show that $t$ is a pure embedding. Let $\varphi(y)$ be a $pp$-formula and suppose that $M_3 \vDash \varphi(m_1 +m_2)$ for $m_1 \in M_1$ and $m_2 \in M_2$. Let $A = \{ m_2\}$ and as before we can find $L_1, L_2, L_3, Q$ and $s: Q \to L_3$ a pure embedding. 

 Since $M_3 \leq_{Pure} L_3$, $L_3 \vDash \varphi(m_1 +m_2)$. Observe that $s([(m_1,m_2)]_Q)= m_1 + m_2$, so $Q \vDash \varphi( [(m_1,m_2)]_Q)$ as $s$ is a pure embedding. Suppose $\varphi(y) = \exists \bar{w}  ( \bigwedge_{i=1}^{k} \left( \sum_{j=1}^{n} r_{i,j} w_j = s_i y \right) )$. Since $Q \vDash \varphi( [(m_1,m_2)]_Q)$, there are $[(p_1^1, p_2^1)]_Q, \cdots, [(p_1^n, p_2^n)]_Q\in Q$ such that \[Q \vDash \bigwedge_{i=1}^{k} \left( \sum_{j=1}^{n} r_{i,j} w_j = s_i y \right) ( [(p_1^1, p_2^1)]_Q, \cdots, [(p_1^n, p_2^n)]_Q,  [(m_1,m_2)]_Q]).\]
 
Then for every $i \in \{1, \cdots k\}$, there is $\ell_i \in M_0$ such that $(\sum_{j=1}^{n} r_{i,j} p_1^j- s_i m_1, \sum_{j=1}^{n} r_{i,j} p_2^j- s_i m_2) = (\ell_i, -\ell_i)$. So 

\[ L_1 \oplus L_2 \vDash \bigwedge_{i=1}^{k} \left( \sum_{j=1}^{n} r_{i,j} w_j = s_i y  + x _i\right)( (p_1^1, p_2^1), \cdots, (p_1^n, p_2^n),  (m_1,m_2), (\ell_1, -\ell_1), \cdots, (\ell_k, -\ell_k))\]

Taking the projection onto $L_1$ and $L_2$ and introducing an existential quantifier, we have that $L_1 \vDash \exists \bar{w} \bigwedge_{i=1}^{k} \left( \sum_{j=1}^{n} r_{i,j} w_j = s_i y  + x_i \right) (m_1, \bar{\ell})$ and $L_2 \vDash \exists \bar{w} \bigwedge_{i=1}^{k} \left( \sum_{j=1}^{n} r_{i,j} w_j = s_i y  + x_i \right) (m_2, -\bar{\ell})$. 

Since $\exists \bar{w} \bigwedge_{i=1}^{k} \left( \sum_{j=1}^{n} r_{i,j} w_j = s_i y  + x_i \right) $ is a $pp$-formula, $m_1, \bar{\ell} \in M_1$ and $M_1\leq_{Pure} L_1$, there are $g_1, \cdots, g_n \in M_1$ such that:

\[ M_1 \vDash \bigwedge_{i=1}^{k} \left( \sum_{j=1}^{n} r_{i,j} w_j = s_i y + x_i \right) (\bar{g}, m_1, \bar{\ell} )  \]

Using again that $ \exists \bar{w} \bigwedge_{i=1}^{k} \left( \sum_{j=1}^{n} r_{i,j} w_j = s_i y  + x_i \right) $ is a $pp$-formula  together with the fact that  $m_2, -\bar{\ell} \in M_2$, $L_2 \leq_{Pure} L_3$, $M_3 \leq_{Pure} L_3$ and $M_2 \leq_{Pure} M_3$, it follows that there are $h_1, \cdots, h_n \in M_2$ such that:

\[  M_2\vDash \bigwedge_{i=1}^{k} \left( \sum_{j=1}^{n} r_{i,j} w_j = s_i y + x_i \right) (\bar{h}, m_2, -\bar{\ell} ).\]
 
 Therefore, \[ P = M_1 \oplus M_2 /\{ (m, -m) : m \in M_0 \}  \vDash \bigwedge_{i=1}^{k} \left( \sum_{j=1}^{n} r_{i,j} w_j = s_i y \right) [ [(g_1, h_1)]_P, \cdots, [(g_n, h_n)]_P,  [(m_1,m_2)]_P].\] Hence $P \vDash \varphi(  [(m_1,m_2)]_P)$. 

\end{proof}

\begin{theorem}\label{s-ind}
If $(K, \M)$ satisfies Hypothesis \ref{hyp2}, then $\dnf$ is a stable independence relation on $\km$. 
\end{theorem}
\begin{proof}
Follows from Fact \ref{weak-stable}, Lemma \ref{local} and Lemma \ref{s-witness}.
\end{proof}

\begin{remark}
 The previous theorem applied to $ \M= Pure$ answers Question 4.23 of \cite{maz2}.
\end{remark}

The next assertion follows from the previous theorem and \cite[3.1]{lrvcell}. 

\begin{cor}\label{rd-co}
$RD$-embeddings are cofibrantly generated in the class of $R$-modules, i.e., they are generated from a set of morphisms by pushouts, transfinite composition and retracts.
\end{cor} 

\begin{remark}
The same result for pure embeddings was recently obtained in \cite{lprv} and the result for embeddings is well-known. 
\end{remark}

\section{Baer-like criterion for relative injective modules} 
 We use the stable independence relation obtained in the previous section to obtain a Baer-like criterion for relative injective modules. 
 
\subsection{Basic notions} We begin by recalling relative injective objects. These were first introduced in \cite{Ma} (see also \cite{AR}).

\begin{defin} Let $(\K, \M)$ be pair such that $\K$ is category and $\M$ a class of morphisms in $\K$. An object $E$ in $\K$ is $\km$-injective if for every morphism $f: A \to B$ in $\M$ and morphism $g: A \to E$ in $\K$ there is morphism $h: B \to E$ in $\K$ such that the following triangle commutes:

\[\begin{tikzcd}
A \arrow[r, "f"] \arrow[d, "g"'] & B  \arrow[ld, "h", dashed] & \hspace{-1cm} \in \M \\
E                                &                            &       
\end{tikzcd}\]
\end{defin}

\begin{remark}
In the rest of the paper $\K$ will always be a class $K$ of modules with morphisms given by $R$-homomorphisms and $\M$ will be a class of $R$-homomorphisms.  Due to this we write $(K, \M)$ instead of $(\K, \M)$ as the category $\K$ is completely determined by the class of modules $K$.
\end{remark}

We give some examples of relative injective modules.

\begin{example} \
\begin{enumerate}
\item  If $(K, \M)=(_{R}Mod, Emb)$, then the $\km$-injective modules are the injective modules.
\item  If $(K, \M)=(_{R}Mod,RD)$, then the $\km$-injective modules are the $RD$-injective modules.
\item  If $(K, \M)=(_{R}Mod, Pure)$, then the $\km$-injective modules are the pure injective modules.
\item  If $(K, \M)=(_{R}Flat, Pure)$, then the $\km$-injective modules are the flat cotorsion modules.
\item If $(K, \M)=(\s\text{-}Tor, Pure)$ where $\s\text{-}Tor$ is the class of $\s$-torsion modules (see \cite{maru}) then the $\km$-injective modules are the $K^{\s\text{-Tor}}$-pure injective  modules (see \cite[3.6]{m-tor}).
\item  If $(K, \M)=(_{R}AbsP, Emb)$ where $_{R}AbsP$ is the class of absolutely pure modules, then  the $\km$-injective modules are the injective modules.
\end{enumerate}
\end{example}

The proof of the  following proposition is standard so we will not give any details. 

\begin{prop} Assume $(K, \M)$ satisfied Hypothesis \ref{hyp2}.
\begin{enumerate}
\item The class of $\km$-injective modules is closed under direct summands and finite direct sums. 
\item If $f: E \to A$ is a $\M$-embedding, $A \in K$, and $E$ is $\km$-injective, then $f[E]$ is a direct summand of $A$ (in $K$), i.e., there is $B \in K$ such that $A = f[E] \oplus B$. 
\end{enumerate}
\end{prop}

It is worth emphasizing that the previous result holds for arbitrary $\M$ as long as (1) and (3) of  Hypothesis \ref{hyp2} hold and $\cm$ contains split embeddings.

Recall that given a module $M$ and a set $I$, $M^{(I)}$ is the direct sum of $I$ many copies of $M$. 

\begin{defin}
$E$ is $\Sigma$-$\km$-injective if and only if $E^{(I)}$ is $\km$-injective for every set $I$. 
\end{defin}

We will show at the end the next subsection  that, under Hypothesis \ref{hyp2}, one only needs to test small sets to determine if a module is $\Sigma$-$\km$-injective.

\subsection{Baer-like criterion} We need the following key result to show a Baer-like criterion for relative injective modules. The result is proved similarly to \cite[8.14]{lrv1}, but both the hypothesis and conclusion are stronger.

\begin{lemma}\label{deco} Assume $(K, \M)$ satisfies Hypothesis \ref{hyp2}.
If $A \leq_\M B$  and $\| B \| = \lambda > \text{card}(R) + \aleph_0$, then there are  $\{ A_i : i < \lambda \}$, $\{ B_i : i < \lambda \}$ increasing and continuous chains such that

\begin{enumerate}
\item For all $i <\lambda$, $A_i \leq_\M  A$, $B_i \leq_\M B$, $A_i \leq_\M B_i$. 
\item For all $i <\lambda$, $\|A _ i \|, \| B_i \|  < \lambda $. 
\item If $i< j $, then $B_i \dnf^{B_j}_{A_i} A_j$.
\item $A=\bigcup_{i < \lambda} A_i $ and  $B=\bigcup_{i < \lambda} B_i $.
\end{enumerate}

\end{lemma}
\begin{proof}
Let $A \leq_\M B$  and $\| B \| > \lambda$. Let $B = \{ b_i : i < \lambda \}$ be an enumeration of $B$. We build  $\{ A_i : i < \lambda \}$, $\{ B_i : i < \lambda \}$ increasing continuous chains by induction such that:

\begin{enumerate}
\item For all $i <\lambda$, $A_i \leq_\M A$, $B_i \leq_\M B$, $A_i \leq_\M B_i$. 
\item For all $i <\lambda$, $\|A _ i \|, \| B_i \|  \leq \text{card}(R) + \aleph_0 +  |i|$. 
\item For all $i <\lambda$, $B_i \dnf^{B}_{A_i} A$
\item For all $i < \lambda$, $b_i \in B_{i+1}$ and if $b_i \in A$, then $b_i \in A_{i+1}$. 
\end{enumerate}

It is easy to show that the above four conditions are enough, so we only need to do the construction.

In the base step, let $A_0$ be the structure obtained by applying the L\"{o}wenheim-Skolem-Tarski axiom to $\emptyset$ in $A$.  Let $B_0 =A_0$. It is easy to see that $A_0, B_0$ are as needed, see for example \cite[3.12]{lrv1}. We do the induction step.

Let $i < \lambda$ be limit ordinal. Let $A_i = \bigcup_{j < i} A_j$ and $B_i = \bigcup_{j < i} B_j$. The only condition that requires an argument is  that $B_i \dnf^{B}_{A_i} A$. By symmetry and the $(<\aleph_0)$-witness property it is enough to show that $ \dnfb{A_i}{A}{D}{B} $ for every $D \subseteq |B_i|$ such that $D$ is finite. Let $D$ be a finite subset of $B_i$, then there is an $j < i$ such that $D \subseteq |B_j|$. As $B_j \dnf^{B}_{A_{j}} A$ by induction hypothesis, $\dnfb{A_{j}}{D}{A}{B}$ by monotonicity. Then $\dnfb{A_i}{D}{A}{B}$ by base monotonicity. Therefore, $ \dnfb{A_i}{A}{D}{B}$ by symmetry. 

Let $i = j +1 <\lambda$. We build increasing chains $\{M_ k: k < \omega\}, \{N_k : k < \omega \}$ such that:
\begin{enumerate}
\item $b_j \in N_0$ and if $b_j \in A$ then $b_j \in M_0$. 
\item $A_{j} \leq_\M M_0$ and $B_j \leq_\M N_0$
\item For all $k < \omega$, $M_k \leq_\M A$ and $M_k \leq_\M N_k\leq_\M B$.
\item For all $k < \omega$, $\| M_k \| \leq \| N_k \| \leq |j+1| + \text{card}(R) + \aleph_0$.
\item For all $k < \omega$,  $\dnfb{M_{k+1}}{N_k}{A}{B}$.
\end{enumerate}

\fbox{Enough} Let $A_{j+1} = \bigcup_{k < \omega} M_k$ and $B_{j+1} = \bigcup_{k < \omega} N_k$. The only condition that requires an argument is  that $B_{j+1} \dnf^{B}_{A_{j+1}} A$, but this can be shown as in the limit case. 

\fbox{Construction} In the base step let $M_0$ be the structure obtained by applying the L\"{o}wenheim-Skolem-Tarski axiom to $A_j \cup (A \cap \{ b_j\})$ in $A$ and $N_0$  be the structure obtained by applying the L\"{o}wenheim-Skolem-Tarski axiom to $M_0 \cup B_j \cup \{ b_j\}$ in $B$. $M_0$ and $N_0$ are as needed as $\LS(\km) = \text{card}(R) + \aleph_0$. We do the induction step.

By induction we have that $N_k \leq_\M B$ and $A \leq_\M B$. Then there are $M^*, N^* \in K$ such that $M^* \leq_\M N^*, A \leq_\M B$, $N_k \leq_\M N^*$, $\| M^* \| \leq  \| N_k \| + \text{card}(R) + \aleph_0$ and $N^* \dnf^B_{M^*}  A$ by Lemma \ref{local}.

Let $M_{k+1}$ be the structure obtained by applying the L\"{o}wenheim-Skolem-Tarski axiom to $M^* \cup M_k$ in $A$. We show that $M_{k+1}$ is as needed. Observe that $\| M_{k+1}\| \leq |j+1| + \text{card}(R) + \aleph_0$ because $\| M^* \| \leq  \| N_k \| + \text{card}(R) + \aleph_0$ and by the induction hypothesis. Moreover, since $N^* \dnf^B_{M^*}  A$ and $N_k \leq_\M N^*$ using monotonicity and base monotonicity one can show that $\dnfb{M_{k+1}} {N_k}{A}{B}$. Let $N_{k+1}$ be the structure obtained by applying the L\"{o}wenheim-Skolem-Tarski axiom to $N_k \cup M_{k+1}$ in $B$. It is easy to see that $N_{k+1}$ is as needed.  \end{proof}

\begin{remark}  Suppose $\Kk$ is an AEC with a stable independence relation $\dnf$ having the $(<\aleph_0)$-witness property. Let $\Kk_{NF}$ be defined in terms of the stable independence relation as in \cite[8.12]{lrv1}. The argument of Lemma \ref{deco} can be used to show that  $\Kk_{NF}$ is an abstract elementary class with the exception that  $\leq_{NF}$ might not refine the substructure relation. 
\end{remark} 

We present the main result of the section. 

\begin{theorem}\label{s-inj}
Assume $(K, \M)$ satisfies Hypothesis \ref{hyp2}. $E$ is a $\K_\M$-injective module if and only if if for every  $f: A \to B$ a $\M$-embedding with $A, B \in K$, $\|A\|, \| B \| \leq  \text{card}(R) + \aleph_0$ and $g: A \to E$ a $R$-homomorphism there is $h : B \to E$ a $R$-homomorphism such that $h \circ f = g$. 
\end{theorem}
\begin{proof}
The forward direction is clear so we prove the backward direction. We show by induction on $\lambda \geq \text{card}(R) + \aleph_0$ that for every $A \leq_\M B$, $g: A \to E$ a $R$-homomorphism and $\| B \| \leq \lambda$ there is $g: B \to E$ extending $f$. This is enough as $K$ is closed under isomorphisms.

The case when $\lambda =  \text{card}(R) + \aleph_0$ is given by assumption, so we do the induction step. Let  $A \leq_\M B$, $g: A \to E$ a $R$-homomorphism and $\| B \| \leq \lambda$. We may assume  $\| B \| = \lambda > \text{card}(R) + \aleph_0$.  Since $A \leq_\M B$, it 
follows from Lemma \ref{deco} that there are  $\{ A_i : i < \lambda \}$, $\{ B_i : i < \lambda \}$ increasing and continuous chains such that:

\begin{enumerate}
\item For all $i <\lambda$, $A_i \leq_\M A$, $B_i \leq_\M B$, $A_i \leq_\M B_i$. 
\item For all $i <\lambda$, $\|A _ i \|, \| B_i \|  < \lambda $. 
\item If $i< j $, then $B_i \dnf^{B_j}_{A_i} A_j$.
\item $A = \bigcup_{i < \lambda} A_i $ and  $B= \bigcup_{i < \lambda} B_i$.
\end{enumerate}

For every $i < \lambda$, let $g_i = g \upharpoonright_{A_i} : A_i \to E$.  We build $\{ h_i : i < \lambda\}$ by recursion on $i < \lambda$:

\begin{enumerate}
\item For all $i < \lambda$, $h_i: B_i \to E$.
\item  For all $i < \lambda$, $g_i \subseteq h_i$
\item If $i < j $, then $h_i \subseteq h_j$. 
\end{enumerate}

This is enough by taking $h = \bigcup_{i < \lambda} h_i: B \to E$. $h$ is a homomorphism with domain $B$ and it extends $g$ by Condition (2). We do the construction.

 \underline{Base step}: Since $A_0 \leq_\M B_0$, $g_0: A_0 \to E$ and $ \| B_0 \| < \lambda$, there is $h_0: B_0 \to E$ extending $g_0$ by the induction hypothesis.

 \underline{Induction step:} If $i$ is limit, let $h_i = \bigcup_{j < i} h_j$. We are left with the case when $i = j +1$ which will requite substantial work.
 
  Assume we have $h_j: B_j \to E$. Let $(P, p_A: A_{j+1} \to P, p_B: B_{j} \to P)$ be the pushout of $A_j \leq_p A_{j+1}, B_j$. 
 
 Since $B_j \dnf^{B_{j+1}}_{A_j} A_{j+1}$, we have the following commutative diagram 
 
 $$\begin{tikzcd}
                                                   &                                                                                  &  & B_{j+1} \\
B_j \arrow[r, "p_B"'] \arrow[rrru, bend left] & { P} \arrow[rru, "t" description] &  &   \\
A_j \arrow[u] \arrow[r]                       & A_{j+1} \arrow[u, "p_A"] \arrow[rruu, bend right]                                      &  &  
\end{tikzcd}$$ where  $t: P \to B_{j+1}$ is a $\M$-embedding. 
 
 Moreover, by Condition (2) $g_{j+1} \upharpoonright_{A_j} = h_j\upharpoonright_{A_j}$, so there is $s: P \to E$ a $R$-homomorphism such that the following diagram commutes:

  \[\begin{tikzcd}
                                                   &                                                                                  &  & E \\
B_j \arrow[r, "p_B"'] \arrow[rrru, bend left, "h_j"] & { P} \arrow[rru, "s" description] &  &   \\
A_j \arrow[u] \arrow[r]                       & A_{j+1} \arrow[u, "p_A"] \arrow[rruu, bend right, "g_{j+1}"]                                      &  &  
\end{tikzcd}\]

Since $\| B_{j+1} \| <\lambda$, $t: P \to B_{j+1}$ is a $\M$-embedding and $s: P \to E$ is a $R$-homomorphism, by induction hypothesis there is $h_{j+1}: B_{j+1} \to E$ such that $h_{j+1} \circ t= s$. Using the above two diagrams one can show that $h_{j+1}$ is as needed.

\end{proof}

\begin{remark} Assume $(K, \M)$ satisfies Hypothesis \ref{hyp2} and let $\K= (K, {}_RMorphisms)$. It follows from Theorem \ref{s-ind} and \cite[3.1]{lrvcell}, that if $\K$ is \emph{cocomplete} then $\M$-morphisms are cofibrantly generated in $\K$ , i..e, they are generated from a set of morphisms by pushouts, transfinite composition and retracts.  If $\M$ is cofibrantly generated in $\K$,  then  there is a Baer-like criterion for $\K_\M$-injectivity, i.e., a cardinal bound to test it.  Nevertheless \cite[3.1]{lrvcell} does not yield our bound of $\text{card}(R)+\aleph_0$.

Moreover, \cite[3.1]{lrvcell} asserts that if $\M$ is cofibrantly generated in $\K$ then $\K_\M$ has a stable independence relation (Theorem \ref{s-ind}). 
\end{remark}

\begin{cor}\
\begin{enumerate}
\item $E$ is injective  if and only if for every  $f: A \to B$ an embedding with $\|A\|, \| B \| \leq  \text{card}(R) + \aleph_0$ and $g: A \to E$ a $R$-homomorphism there is $h : B \to E$ such that $h \circ f = g$. 
\item $E$ is $RD$-injective  if and only if for every  $f: A \to B$ a $RD$-embedding with $\|A\|, \| B \| \leq  \text{card}(R) + \aleph_0$ and $g: A \to E$ a $R$-homomorphism there is $h : B \to E$ such that $h \circ f = g$. 
\item $E$ is pure injective if and only if for every  $f: A \to B$ a pure embedding with $\|A\|, \| B \| \leq  \text{card}(R) + \aleph_0$ and $g: A \to E$ a $R$-homomorphism there is $h : B \to E$ such that $h \circ f = g$. 
\item Let $F$ be a flat module. $F$ is cotorsion if and only if for every  $f: A \to B$ a pure embedding with $A, B$ flat, $\|A\|, \| B \| \leq  \text{card}(R) + \aleph_0$ and $g: A \to F$ a $R$-homomorphism there is $h : B \to F$ such that $h \circ f = g$. 
\item  Let $T$ be a $\s$-torsion module. $T$ is $K^{\s\text{-Tor}}$-pure injective if and only if for every  $f: A \to B$ a pure embedding with $A, B$ $\s$-torsion modules, $\|A\|, \| B \| \leq  \text{card}(R) + \aleph_0$ and $g: A \to T$ a $R$-homomorphism there is $h : B \to T$ such that $h \circ f = g$. 
\end{enumerate}
\end{cor}
\begin{proof} 
(1) through (3) and (5) follow directly by applying the previous result to $({}_RMod, Emb)$, $({}_RMod, RD)$, $({}_RMod, Pure)$ and $(\s\text{-Tor}, Pure)$ respectively. 

(4) follows by applying the  previous result to $(R\text{-Flat}, Pure)$ and then using the fact that cotorsion flat modules are the $\km$-injectives in the class of flat modules with pure embeddings. \end{proof}

\begin{remark}   \
\begin{enumerate}
\item of the previous corollary is a weakening of Baer's Criterion \cite{baer}.
\item of the previous corollary is new.
\item  of the previous corollary was obtained using algebraic methods in \cite[1.3]{satr}.
\item of the previous corollary can be obtain using a similar argument to that of \cite[1.1]{satr} using that the class of flat modules is $(\text{card}(R) + \aleph_0)^+$-deconstructible. This is last result is a deep result in module theory (see for example \cite[6.17]{gotr}).
\item of the previous corollary is new.
\end{enumerate}

It is worth mentioning that the methods used to obtain the previously known results are different for each class of modules. 
\end{remark}

\begin{lemma}
Assume $(K, \M)$ satisfies Hypothesis \ref{hyp2}. $E$ is a $\Sigma$-$\K_\M$-injective module if and only if $E^{(\text{card}(R) + \aleph_0)}$ is a $\K_\M$-injective module.
\end{lemma}
\begin{proof}
We only need to show the backward direction. Let $I$ be a set, we show that $E^{(I)}$ is $\K_\M$-injective. If $|I| \leq \text{card}(R) + \aleph_0$ there is nothing to show as  relative injectives are closed under direct summands, so assume that $|I| > \text{card}(R) + \aleph_0$. 
By Theorem \ref{s-inj} it is enough to consider $A \leq_\M B$ and $g: A \to E^{(I)}$ a $R$-homomorphism with $\| B \| \leq \text{card}(R) + \aleph_0$. Since $\|A\| \leq \text{card}(R) + \aleph_0$, there is $J \subseteq I$ such that $|J| = \text{card}(R) + \aleph_0$ and $\nu_J \circ \pi_J (g(a)) =g(a)$ for every $a \in A$ where   $\pi_J: E^{(I)} \to E^{(J)}$ is the canonical projection map and $\nu_J: E^{(J)} \to E^{(I)}$ the canonical inclusion. Since $ E^{(J)}$ is $\K_\M$-injective by assumption, there is $h: B \to E^{(J)}$ such that $\pi_J \circ g  = h\upharpoonright_A$.  It is clear that $\nu_J \circ h: B\to E^{(I)}$ is a required. 
\end{proof}

As an immediate corollary we get. 
\begin{cor}
Let $F$ be a flat module. $F$ is $\Sigma$-cotorsion if and only if $F^{(\text{card}(R) + \aleph_0)}$ is cotorsion.
\end{cor}

\begin{remark}
The previous result was obtained in  \cite{gh2} for countable rings using model-theoretic methods  and in \cite[3.8]{sast} using set-theoretic methods without the assumption that $F$ is flat.
\end{remark}

It is known that  $\Sigma$-pure injective modules \cite[2.11]{prest}, flat $\Sigma$-cotorsion modules \cite{sast} and $\Sigma$-$K^{\s\text{-Tor}}$-pure injective are closed under pure submodules \cite[3.14]{m-tor}.  So it is natural to ask: 

\begin{question} Assume $\M = Pure$ and $(K, \M)$ satisfies Hypothesis \ref{hyp2}. 
Are $\Sigma$-$\K_\M$-injective modules closed under pure submodules? 
\end{question}

\subsection{Saturated models} We show a relation between saturated models and $\km$-injective modules. We  use this relation to show that there are enough $\km$-injectives. Recall that a submodule $A$ of $B$ is a retract if there is $\pi: B \to A$ such that  $\pi\rest_A = \id_A$.

 \begin{lemma}\label{sat} Assume $(K,\M)$ satisfies Hypothesis \ref{hyp2} and let $E \in K$.
$E$ is a $\K_\M$-injective module if and only if $E$ is a retract of a $(\text{card}(R) + \aleph_0)^+$-saturated model in $\K_\M$. 
\end{lemma}
\begin{proof}
$\Rightarrow$: Let $E$ be $\K_\M$-injective, then there is $M$ $(\text{card}(R) + \aleph_0)^+$-saturated in $\K_\M$ such that $E \leq_\M M$ (see for example \cite[6.7]{grossberg2002}). Since $E$ is $\K_\M$-injective, $E$ is a direct summand of $M$. Hence $E$ is a retract of $M$.

$\Leftarrow$: Let $E \leq_\M M$ such that  $M$ is $(\text{card}(R) + \aleph_0)^+$-saturated and $\pi: M \to E$ with $\pi\rest_E = \id_E$.

 We use Theorem \ref{s-inj} to show that $E$ is $\km$-injective. Let $ A  \leq_\M B$ with $A, B \in K$, $\|A\|, \| B \| \leq  \text{card}(R) + \aleph_0$ and $g: A \to E$ a $R$-homomorphism. In particular, $g: A \to M$ is a $R$-homomorphism. As $\|A\| \leq  \text{card}(R) + \aleph_0$, using the L\"{o}wenheim-Skolem-Tarski axiom in $\km$ there is $M_0 \leq_\M M$ such that $g[A] \subseteq |M_0|$ and $\| M_0 \| \leq \text{card}(R) + \aleph_0$. So $g: A \to M_0$ is  a $R$-homomorphism.

Take the pushout of $(i: A \hookrightarrow  B, g: A \to M_0)$ in the category of modules:
$$\begin{tikzcd}
B \arrow[r, "\bar{g}"]           & P                      \\
A \arrow[u, "i"] \arrow[r, "g"'] & M_0 \arrow[u, "\bar{i}"']
\end{tikzcd}$$
 
 Since $i: A \hookrightarrow  B$ is a $\M$-embedding as $A \leq_\M, B$, $\bar{i}: M_0 \to P$ is $\M$-embedding by Lemma \ref{push}. Moreover $P \in K$ by Lemma \ref{push} and $\| P \| \leq \text{card}(R) + \aleph_0$ as $\|B\|, \| M_0\| \leq \text{card}(R) + \aleph_0$.
 
 Since $M$ is $(\text{card}(R) + \aleph_0)^+$-saturated, $M_0 \leq_\M M$ and $\bar{i}: M_0 \to P$ is a $\M$-embedding with $\|M_0\|, \| P\| \leq \text{card}(R) + \aleph_0$ and $M_0, P \in K$, there is $s: P \to M$ a $\M$-embedding  such that $s \circ \bar{i}= \id_{M_0}$. 
 
 Let $h := \pi \circ s \circ \bar{g} : B \to E$. Using the pushout diagram, that $g[A] \subseteq |M_0|$, that $s \circ \bar{i}= \id_{M_0}$ and that $\pi\upharpoonright_E = \id_E$, one can show that $g = h\rest_A$. Hence $E$ is $\K_\M$-injective. \end{proof}

\begin{cor}\label{sat-inj} Assume $(K,\M)$ satisfies Hypothesis \ref{hyp2}. 
If $M \in K$ is $\lambda$-saturated for some $\lambda \geq (\text{card}(R) + \aleph_0)^+$, then $M$ is $\K_\M$-injective. 
\end{cor}

\begin{cor} Assume $(K,\M)$ satisfies Hypothesis \ref{hyp2}. 
$\K_\M$ has enough $\K_\M$-injectives, i.e., if $A \in K$ then there is $B \in K$ such that $A \leq_\M B$ and $B$ is $\km$-injective.
\end{cor}
\begin{proof}
Every $A \in K$ is a $\K_\M$-submodule of an $(\text{card}(R) + \aleph_0)^+$-saturated model in $\K_\M$ (see for example \cite[6.7]{grossberg2002}).
\end{proof}

\section{Superstability and noetherian categories}

In this section we will show, assuming Hypothesis \ref{hyp2}, that $\km$ is superstable if and only if $(K, \M)$ is a noetherian category.

\subsection{Noetherian categories} The class of $\km$-injective modules might not be closed under arbitrary direct sums.  If the class of   $\km$-injective modules is closed under arbitrary direct sums, we will say that $(K, \cm)$ is noetherian following \cite[p. 123]{ste}.

\begin{defin} Assume $(K, \M)$ is a pair such that $K$ is  a class of $R$-modules  and $\M$ is a class of $R$-homomorphisms for a fixed ring $R$.
 $(K, \M) $ is \emph{noetherian} if and only if every direct sum of $\km$-injectives is $\km$-injective.\footnote{For an arbitrary category $\K$ with coproducts and $\M$ a class of morphisms in $\K$, we say that $(\K, \M)$ is noetherian if every coproduct of  $\km$-injectives is $\km$-injective.} 
\end{defin}

The classical result of Bass-Papp (see for example \cite[3.46]{lam2}) states that $(_{R}Mod, Emb)$ is noetherian if and only if $R$ is a left noetherian ring.

\begin{prop}\label{add-lim}  Assume $(K, \M)$ satisfies Hypothesis \ref{hyp2}. $(K,\cm)$ is noetherian  if and only if the class of $\km$-injectives is closed under directed colimits in $\km$.
\end{prop}
\begin{proof}
$\Rightarrow$: 
It suffices to prove that the class of $\km$-injectives is closed under colimits of smooth chains of $\cm$-embeddings \cite[1.7]{AR}. Since every $\cm$-embedding $f: K\to L$ where $K$ is $\km$-injective splits, $f$ is the injection $K\to L=K\oplus C$ for $C$ a $\km$-injective. Hence colimits of smooth chains are direct sums of $\km$-injectives.

$\Leftarrow$:  A direct sum is a directed colimit of finite direct sums and split monomorphisms. Since finite direct sums of $\km$-injectives are $\km$-injective,  $(K,\cm)$ is noetherian.\end{proof}

It is worth emphasizing that the previous result holds for arbitrary $\M$ as long as (1) and (3) of  Hypothesis \ref{hyp2} hold and $\cm$ contains split embeddings.

\subsection{Main equivalence} The following result is important. We do not provide any details as the proof in the case when  $\M = Pure$ was obtained in  \cite[4.17]{maz2} and that proof carries over to $\M = Emb$ and $\M = RD$ by Fact \ref{weak-stable} and Lemma \ref{local}.

\begin{theorem}\label{statf}
 Assume $(K, \M)$ satisfies Hypothesis \ref{hyp2}. If $\lambda^{\text{card}(R) + \aleph_0}=\lambda$, then $\km$ is $\lambda$-stable.
\end{theorem}

For completeness we record the following result which can be obtained as in \cite[4.21]{maz2}, the result when $\M = Pure$ is known. 

\begin{lemma}\label{tame} 
 Assume $(K, \M)$ satisfies Hypothesis \ref{hyp2}. Then $\km$ is $(\text{card}(R) + \aleph_0)$-tame. 
\end{lemma}

Throughout the rest of this section we say that $f: A \to B$ is a $\km$-embeddings if $f:A \to B$ is a $\M$-embedding and $A, B \in K$. This is consistent with the model theoretic notation used for abstract elementary classes which was introduced in Section 2.

We turn towards understanding superstability. We begin by showing that $\km$-injective modules in noetherian categories satisfy a Schr\"{o}der-Bernstein type property. 

\begin{lemma}\label{csb} 
Assume $(K, \M)$ satisfies Hypothesis \ref{hyp2} and $(K, \M)$ is noetherian. If $f: A \to B$, $g: B \to A$ are $\M$-embeddings and $A, B $ are $\km$-injective, then $A$ is isomorphic to $B$.
\end{lemma}
\begin{proof}
We may assume without lost of generality that $A \leq_\M B$. Since $A$ is $\km$-injective, there is $C \in K$ such that $B = C\oplus  A$.  

We show by induction that  $\bigoplus_{k \leq n } g^{k}[C] \leq_\M B$ for every $n < \omega$ where we are taking internal direct sums in $B$. The base step is clear as $g^0[C]=C$ is a direct summand of $B$, so we do the induction step. By induction hypothesis $\bigoplus_{k \leq n } g^{k}[C] \leq_\M B$. Since $g:  B \to A$ is a $\K_\M$-embedding  $g[ \bigoplus_{k \leq n } g^{k}[C]  ] \leq_\M  A$, then $C \oplus g[ \bigoplus_{k \leq n } g^{k}[C]  ]  \leq_\M C \oplus A$ by Proposition \ref{v-basic}. Substituting $B$ by $C \oplus A$ we get that $ \bigoplus_{k \leq n +1 } g^{k}[C] = C \oplus  \bigoplus_{k \leq n } g^{k+1}[C]  \leq_\M C \oplus A = B$. 

Let $C^* = \bigoplus_{n < \omega} g^{n}[C]$. Since $\{ \bigoplus_{k \leq n } g^{k}[C] : n < \omega \}$ is an increasing chain in $\K_\M$ bounded by $B$, we get that $ C^* = \bigcup_{n < \omega} (\bigoplus_{k \leq n } g^{k}[C] ) \leq_\M B$. Since $g: B \to A$ is a $\km$-embedding $g[C^*]= \bigoplus_{k \geq 1 } g^{k}[C] \leq_\M A$. Observe that $C$ is $\km$-injective because it is a direct summand of $B$, then $g[C^*]= \bigoplus_{k \geq 1 } g^{k}[C]$ is $\km$-injective because $(K,\M)$ is noetherian.  So there is $D\in K$ such that $A = g[C^*] \oplus D$. Then \[B = C \oplus A = C \oplus (g[C^*] \oplus D)= (C \oplus g[C^*]) \oplus D.\] As $C \oplus g[C^*] = C^*$ and $g: C^* \cong g[C^*]$ we get that  $(C \oplus g[C^*]) \oplus D  = C^* \oplus D \cong g[C^*] \oplus D$. Finally, as  $g[C^*] \oplus D=A$, we conclude that $B \cong A$.\end{proof}

The previous proof has some similarities with the proof of \cite[2.4]{gks}. \cite{gks} generalizes the classical result for injective modules of Bumby \cite{bumby}.

In most of the examples given in Example \ref{ex1}, the previous result holds even if  $(K, \M)$ is not noetherian, so it is natural to ask:
\begin{question}\label{q3} 
Does the previous result still hold even if $(K, \M)$ is not noetherian? 
\end{question}


Before proving the equivalence between superstable AECs and noetherian categories, we need to  understand  limit models and universal extensions. 

\begin{prop}\label{lim-inj}  Assume $(K, \M)$ satisfies Hypothesis \ref{hyp2}.  Let  $\lambda  \geq (\text{card}(R) + \aleph_0)^+$ be an  infinite cardinal and  $\kappa$ be a regular cardinal such that $\lambda \geq \kappa \geq (\text{card}(R) + \aleph_0)^+$. If $M$ is a $(\lambda, \kappa)$-limit model in $\km$, then $M$ is $\km$-injective. 
\end{prop} 
\begin{proof}
$M$ is a $\kappa$-saturated model by Fact \ref{sat-lim}. Hence $M$ is $\km$-injective by Corollary \ref{sat-inj}.  
\end{proof}

\begin{lemma}\label{uni+} Assume $(K, \M)$ satisfies Hypothesis \ref{hyp2}, $\lambda$ is an infinite cardinal  and $M \in K$ with $\| M \| \leq \lambda$.
If $M$ is $\km$-injective  and $N$ is  universal in $(\km)_\lambda$ , then $M \oplus N$ is $\lambda$-universal over $M$.
\end{lemma} 
\begin{proof}
Let $L \in (\K_\M)_{\leq \lambda}$ such that $M \leq_\M L$. Then there is $L^* \in K$ such that $ L= M \oplus L^*$ as $M$ is $\km$-injective. Since $N$ is universal in $(\K_\M)_\lambda$ there is $g: L^* \to N$ a $\K_\M$-embedding. Then $f:  L= M \oplus L^* \to M \oplus N$ given by $f(m + l)= m + g(l)$ is a $\km$-embedding by Proposition \ref{v-basic}.  \end{proof}



The next result is the main technical result of this section.

\begin{lemma}\label{sum-lim}   Assume $(K, \M)$ satisfies Hypothesis \ref{hyp2}.  Let  $\lambda  \geq (\text{card}(R) + \aleph_0)^+$  such that $\K_\M$ is $\lambda$-stable.
If $M$ is the $(\lambda, (\text{card}(R) + \aleph_0)^+)$-limit model and $(K, \M)$ is noetherian or $\K_\M$ has uniqueness of limit models of cardinality $\lambda$, then $M^{(\alpha)}$ is the $(\lambda, \alpha)$-limit model and $M^{(\alpha)}$ is $\km$-injective for every limit ordinal $\alpha < \lambda^+$.
\end{lemma}
\begin{proof}
Let  $\alpha < \lambda^+$ be a limit ordinal and consider $\{M^{(\gamma)} : 0 < \gamma \leq  \alpha \}$, we show by induction on $0< \gamma \leq  \alpha$ that:

\begin{enumerate}
\item $M^{(\gamma)}$ is $\km$-injective.
\item $M^{(\gamma +1)}$ is universal over $M^{(\gamma)}$.
\end{enumerate}
Before we do the proof, observe that this is enough as $\{M^{(\gamma)} : 0 < \gamma <  \alpha \}$ witnesses that $M^{(\alpha)}$ is the $(\lambda, \alpha)$-limit model and by taking $\gamma = \alpha$ we get that $M^{(\alpha)}$ is $\km$-injective. We do the proof by induction.

\underline{Base:} $M$ is $\km$-injective by Lemma \ref{lim-inj}. Moreover, $M \oplus M$ is universal over $M$ by Lemma \ref{uni+}.

\underline{Induction step:} If $\gamma = \beta +1$, then $M^{(\beta +1)}$ is $\km$-injective, since  $M^{(\beta)}, M$ are $\km$-injective and $\km$-injectives are closed under finite direct sums. Moreover,  (2) follows from Lemma \ref{uni+}.

If $\gamma$ is a limit ordinal, then we divide the argument into two cases depending on whether $(K, \M)$ is noetherian or $\km$ has uniqueness of limit models:

\begin{itemize}
\item \underline{Case 1:} Assume $(K, \M)$ is  noetherian. Then $M^{(\gamma)}$ is $\km$-injective because $M$ is $\km$-injective. Moreover, $M^{(\gamma)} \oplus M$ is universal over $M^{(\gamma)}$ by Lemma \ref{uni+}.
\item \underline{Case 2:} Assume $\K_\M$ has uniqueness of limit models of cardinality $\lambda$. Then consider $\{ M^{(\beta)} : 0 < \beta < \gamma \}$. It is clear that it is an increasing and continuous chain in $(\K_\M)_\lambda$ such that $\bigcup_{\beta < \gamma }M^{(\beta)} = M^{(\gamma)}$. Moreover, by induction hypothesis $M^{(\beta + 1)}$ is universal over $M^{(\beta)}$ for $\beta < \gamma$. Therefore, $\{ M^{(\beta)} : 0<\beta < \gamma \}$ witnesses that $M^{(\gamma)}$ is a $(\lambda, \gamma)$-limit model.  Then by uniqueness of limit models of size $\lambda$, $M^{(\gamma)}$ is isomorphic to $M$. We know that $M$ is $\km$-injective, hence $M^{(\gamma)}$ is a $\km$-injective. The argument that $M^{(\gamma+ 1)}$ is universal over $M^{(\gamma)}$ is the same as that of Case 1. 
\end{itemize} \end{proof}
 
 A similar argument to the one above gives us the following result.

\begin{cor}\label{omega} Assume $(K, \M)$ satisfies Hypothesis \ref{hyp2}. Let $\lambda \geq \text{card}(R) + \aleph_0$. If $M$ is $\km$-injective, $\| M \| =\lambda$ and $M$ is universal in $(\K_\M)_\lambda$, then $M^{(\omega)}$ is the $(\lambda, \omega)$-limit model.
\end{cor}

We prove the main result of this section.

\begin{theorem}\label{equiv}  Assume $(K, \M)$ satisfies Hypothesis \ref{hyp2}. The following are equivalent.
\begin{enumerate}
\item $(K, \M)$ is noetherian.
\item $\K_\M$ is superstable.
\end{enumerate}
\end{theorem} 
\begin{proof}
(1) $\Rightarrow$ (2):  Let $\lambda_0$ be the smallest $\lambda \geq  (\text{card}(R) + \aleph_0)^+$ such that $\K_\M$ is $\lambda$-stable. This cardinal exists as $\K_\M$ is stable by Theorem \ref{statf}.

We show that $\K_\M$ has uniqueness of $\lambda$-limit models for every $\lambda \geq \lambda_0$. The proof is divided into two claims. In the first claim we show existence of limit models and in the second claim we show uniqueness.  
 
\underline{Claim 1}:  $\K_\M$ is $\lambda$-stable for every $\lambda \geq \lambda_0$.

\underline{Proof of Claim}: The proof is done by induction. The base step is clear so we prove the induction step. Let $\lambda$ be an infinite cardinal such that $\K_\M$ is $\kappa$-stable for every $\kappa \in [\lambda_0, \lambda)$. Let $cf(\lambda)=\theta$ and $\{ \mu_i < \lambda : i < \theta\}$ be an increasing and continuous chain of cardinals such that $sup_{i < \theta} \mu_i = \lambda^{-}$ and $\mu_i \geq |i| + \lambda_0$ for every $i < \theta$.\footnote{Given a cardinal $\theta$, $\theta^{-} = \eta$ if $\theta = \eta^+$ and $\theta^{-}=\theta$ if $\theta$ is a limit cardinal.} 

Let $N_i$ be the $(\mu_i , (\text{card}(R) + \aleph_0)^+)$-limit model for every $i < \theta$. $N_i$ exists because $\K_\M$ is $\mu_i$-stable by the induction hypothesis.  Let $N = \bigoplus_{i < \theta} N_i$. Since $(K, \M)$ is noetherian and each $N_i$ is $\km$-injective by Proposition \ref{lim-inj}, it follows that $N$ is $\km$-injective.

It is clear that $\| N\| = \lambda$. Moreover, using Lemma \ref{uni+} and doing a similar argument to that of \cite[3.18]{kuma}  one can show that $N$ is universal in $(\K_\M)_\lambda$. Then $N^{(\omega)}$ is a $(\lambda, \omega)$-limit model by Corollary \ref{omega}. Therefore, $\K_\M$ is $\lambda$-stable. $\dagger_{\text{Claim 1}}$
 
 \underline{Claim 2}: There is at most one $\lambda$-limit model for every $\lambda \geq \lambda_0$. 
 
 \underline{Proof of Claim}:  It follows from Lemma \ref{sum-lim} that every $\lambda$-limit model is $\km$-injective. This is enough since $\lambda$-limit models are universal in $(\K_\M)_\lambda$ and by Lemma \ref{csb}.$\dagger_{\text{Claim 2}}$

(2) $\Rightarrow$ (1):  Let $\{ M_\alpha : \alpha < \kappa \}$ be such that $M_\alpha$ is $\km$-injective for every $\alpha < \kappa$.

Let $\lambda > \text{sup}_{\alpha < \kappa} \| M_\alpha \| + (\text{card}(R) + \aleph_0)^+$ be such that $\K_\M$ has uniqueness of limit models of cardinality $\lambda$ and let $N$ be the $(\lambda , (\text{card}(R) + \aleph_0)^+)$-limit model. Since $N$ satisfies the hypothesis of Lemma \ref{sum-lim}, $N^{(\kappa)}$ is the $(\lambda, \kappa)$-limit model and is $\km$-injective.
 
 Since $N$ is universal in $(\K_\M)_\lambda$ by Remark \ref{uni-lim} and each $M_{\alpha}$ is $\km$-injective with $\|M_{\alpha}\| \leq \lambda$, for every $\alpha < \kappa$ there is a $N_\alpha$ such that $N \cong M_\alpha \oplus N_\alpha$. Therefore, we get that:
 
 \[ N^{(\kappa)} \cong \bigoplus_{\alpha < \kappa} (M_\alpha \oplus N_\alpha) \cong  (\bigoplus_{\alpha < \kappa} M_\alpha) \oplus (\bigoplus_{\alpha < \kappa}  N_\alpha)\]

Since $N^{(\kappa)}$ is $\km$-injective and $\km$-injective modules are closed under direct summands, it follows that $\bigoplus_{\alpha < \kappa} M_\alpha$ is $\km$-injective. Therefore, $(K, \M)$ is noetherian. \end{proof}

\begin{remark} Similar results have been obtained for certain classes of modules with embeddings \cite[\S 3]{maz1} and with pure embeddings \cite[\S 4]{maz1}, \cite[\S 3]{m3} \cite[\S 3]{maz2}, \cite[\S 4]{m-tor}, but the result is new in this generality even when $\M$ is the class of pure embeddings and for the specific case of $(_{R}Mod, RD)$.
\end{remark}

\subsection{Characterizing rings via superstability} We characterize several classes of ring via superstability using Theorem \ref{equiv}. Most of the results had already been obtained in  \cite{maz1} and \cite{m3}, but Corollary \ref{rd-ring} is new.

We obtain a characterization of noetherian rings via superstability. Finer results in this direction were obtained in \cite[3.12]{maz1}.  Recall that a ring $R$ is left noetherian if direct sums of injective modules are injective. 

\begin{cor} $(_{R}Mod, \leq_{Emb})$ is superstable if and only if $R$ is left noetherian. 
\end{cor}

We obtain a characterization of finite products of finite rings and artinian valuation rings. 

\begin{cor}\label{rd-ring} Assume $R$ is a commutative ring. $(_{R}Mod, \leq_{RD})$ is superstable if and only if $R$ is a  finite product of finite rings and artinian valuation rings.
\end{cor}
\begin{proof}
Follows from \cite[2.1]{cou}. 
\end{proof}

We obtain a characterization of pure-semisimple rings via superstability. Finer results in this direction were obtained in \cite[4.28]{maz1}. Recall that a ring $R$ is left pure-semisimple if every module is pure injective. 
\begin{cor}\label{pss} $(_{R}Mod, \leq_{Pure})$ is superstable if and only if $R$ is left pure-semisimple.
\end{cor}
\begin{proof}
The backward direction is clear and the forward direction follows from the fact that $\Sigma$-pure injective modules are closed under pure submodules (see for example \cite[2.11]{prest}). 
\end{proof}

We obtain a characterization of perfect rings via superstability. Finer results in this direction were obtained in \cite[3.15]{m3}. Recall that a ring $R$ is left perfect if every flat module is a cotorsion module. 

\begin{cor}\label{per} $(_{R}Flat, \leq_{Pure})$  is superstable if and only if $R$ is left perfect. \end{cor}
\begin{proof}
The backward direction is clear and the forward direction follows from the fact that $\Sigma$-cotorsion modules are closed under pure submodules by \cite{sast}. 
\end{proof}

\subsection{Another equivalent condition} We begin by recalling the notion of a $\lambda$-pure embedding. The definition will mention that some objects are $\lambda$-presentable. This is a notion of size close to cardinality.  See \cite[\S 1 and \S 2]{AR} for more details.

\begin{defin} Let $\K$ be a category and $\lambda$ be a regular cardinal. A morphism $f:A \to B$ is a $\lambda$-pure embedding provided that in each commutative square:

$$\begin{tikzcd}
A' \arrow[d, "u"'] \arrow[r, "f'"] & B' \arrow[d, "v"] \\
A \arrow[r, "f"]                   & B                
\end{tikzcd}$$ with $A'$ and $B'$ $\lambda$-presentable, there is $g: B' \to A$ such that $g \circ f' = u$. 
\end{defin}

In the category of $R$-modules with $R$-homomorphism an embedding is an $\aleph_0$-pure embedding if and only if it is a pure embedding (see for example \cite[2.1.7]{prest09}).

\begin{theorem}
Assume $(K, \M)$ satisfies Hypothesis \ref{hyp2}. The following are equivalent.
\begin{enumerate}
\item $(K, \M)$ is noetherian.
\item $\K_{\M\text{-inj}} = (\km\text{-injective modules}, \leq_\M)$ is an abstract elementary class.
\end{enumerate}
\end{theorem} 
\begin{proof}
The backward direction is clear as split embeddings are contained in $\M$ and finite direct sums of $\km$-injective modules are $\km$-injective modules.

We show the forward direction. The only two axioms that require an argument are the L\"{o}wenheim-Skolem-Tarski axiom and the Tarski-Vaught axioms. The Tarski-Vaught axioms follow from the assumption that $(K, \M)$ is noetherian and Proposition \ref{add-lim}. So we show the  L\"{o}wenheim-Skolem-Tarski axiom 

 We associate to $\K_{\M\text{-inj}}$ a category which objects are the $\km\text{-injective modules}$ and which arrows are the $\cm$-embeddings. For simplicity we will denote this category by $\K_{\M\text{-inj}} $. We do the same for $\km$.

\underline{Claim}: $\K_{\M\text{-inj}}$ is an accessible category.

\underline{Proof of Claim}: Observe that $\km$ is an accessible category by \cite[4.3]{mu}. Since $(K, \M)$ is noetherian, $\K_{\M\text{-inj}}$ is closed under directed colimits in $\km$ by Proposition \ref{add-lim}. Moreover, $\K_{\M\text{-inj}}$  is a full subcategory of $\km$, hence $\K_{\M\text{-inj}}$ is accessibly embedded into $\km$. To show that $\K_{\M\text{-inj}}$ is an accessible category it is enough to show that  $\K_{\M\text{-inj}}$ is closed in $\km$ under $(\text{card}(R)+\aleph_0)^+$-pure subobjects by \cite[2.36]{AR}.

Let $h:M\to N$ be a $(\text{card}(R)+\aleph_0)^+$-pure embedding in $\km$  with $N$ a $\km$-injective module. We use Theorem \ref{s-inj} to show that $M$ is $\km$-injective. Let $ A  \leq_\M B$ with $A, B \in K$, $\|A\|, \| B \| \leq  \text{card}(R) + \aleph_0$ and $u: A \to M$ be a $R$-homomorphism. 

Observe that $h \circ u: A \to N$ is a $R$-homomorphism, so there is $v:  B \to N$ a $R$-homomorphism such that $h \circ u  = v\rest_A$ because $N$ is $\km$-injective.  As $\|A\| \leq  \text{card}(R) + \aleph_0$, using the L\"{o}wenheim-Skolem-Tarski axiom in $\km$ there is $M_0 \leq_\M M$ such that $u[A] \subseteq |M_0|$ and $\| M_0 \| \leq \text{card}(R) + \aleph_0$. So $u: A \to M_0$ is  a $R$-homomorphism. Doing a similar argument in $N$ there is $N_0 \leq_\M N$ such that $h_0=h\rest_{M_0}: M_0 \to N_0$ and $v: B \to N_0$ are $R$-homomorphisms. So we get the following commutative diagram:

$$\begin{tikzcd}
A \arrow[r, "i_A"] \arrow[d, "u"']         & B \arrow[d, "v"]               \\
M_0 \arrow[r, "h_0"] \arrow[d, "i_{M_0}"'] & N_0 \arrow[d, "i_{N_0}"] \\
M \arrow[r, "h"]                                 & N                             
\end{tikzcd}$$ where all the inclusions are $\km$-embeddings. Observe that $h_0: M_0 \to N_0$ is a $\km$-embedding as $h: M \to N$ is a $\km$-embedding and  $M_0 \leq_\M M$. Moreover, $M_0$ and $N_0$ are $(\text{card}(R)+\aleph_0)^+$-presentable in $\km$ as $\| M_0 \|, \| N_0\| \leq \text{card}(R) + \aleph_0$ and by \cite[4.2]{mu}. As $h: M \to N$ is a $(\text{card}(R)+\aleph_0)^+$-pure embedding, there is $k: N_0 \to M$ such that $k \circ h_0 = i_{M_0}$. 

Let $t:= k \circ v: B \to M$. Using the top square of the above diagram and that $k \circ h_0 = i_{M_0}$, it follows that $u = t\rest_A$. Therefore, $ M$ is $\km$-injective. $\dagger_{\text{Claim}}$

Since $\K_{\M\text{-inj}}$ and $\km$ are accessible categories and $\K_{\M\text{-inj}}$ is accessibly embedded into $\km$, it follows from \cite[2.19]{AR} that there is a cardinal $\lambda \geq \LS(\km)$ such that $\K_{\M\text{-inj}}$ and $\km$ are  $\lambda^+$-accessible categories and $\lambda^+$-presentable objects are preserved between $\K_{\M\text{-inj}}$ and $\km$. Then the same argument as that given for the L\"{o}wenheim-Skolem-Tarski axiom in \cite[4.5]{mu} can be used to show that $\LS(\K_{\M\text{-inj}}) \leq \lambda$. The argument presented in \cite[4.5]{mu} can be carried out here as for every cardinal $\theta \geq \LS(\km)$, $M$ is  $\theta^+$-presentable in $\km$ if and only if $\| M \| \leq \theta$ since $\km$ is an AEC and by \cite[4.2]{mu}. \end{proof}

\begin{remark}
The above equivalence was noticed for injective modules and pure injective modules in \cite[3.24, 3.27]{maz2}.
\end{remark} 

\subsection{Extending the framework} The main result of this section (Theorem \ref{equiv}) shows that there is a deep connection between noetherian categories and superstable AECs. In this subsection we give a weaker set of assumptions under which one can prove the equivalence. Since the proofs are the same as those of the previous subsection by replacing $(\text{card}(R) + \aleph_0)^+$ by  $\lambda_{\km}$ we do not provide any details. 

\begin{hypothesis}\label{e-h}
Let $\K_\M=(K, \leq_\M)$ be an AEC with $K \subseteq R\text{-Mod}$ for $R$ a fixed ring $R$ such that:
\begin{enumerate}
\item  $K$ is closed under direct sums and direct summands.
\item Split monomorphims are contained in $\M$, i.e.,  if $A \in K$ is a direct summand of $B\in K$, then $A \leq_\M B$.
\item If $A \leq_\M C$ and $B \leq_\M D$, then $A \oplus B \leq_\M C \oplus D$.
\item $\K_\M$ has the amalgamation property.
\item  $\K_\M$ is tame and stable.\label{h-sta}
\item There is a regular cardinal $\lambda \geq \LS(\K_\M)$ such that if $M$ is $\lambda$-saturated, then $M$ is $\K_\M$-injective.  We write $\lambda_{\km}$ for the smallest such cardinal.\label{inje}

\end{enumerate}
\end{hypothesis}

\begin{example}\label{exam}\
\begin{enumerate}
\item If  $(K, \M)$ satisfies Hypothesis \ref{hyp2}, then $(K, \leq_\M)$ satisfies Hypothesis \ref{e-h}. $(K, \leq_\M)$ is an AEC by Lemma \ref{aec}.  Conditions (1) to (3) are Proposition \ref{v-basic}, Condition (4) is Proposition \ref{ap}, Condition (5) is Theorem \ref{statf} and Lemma \ref{tame} and Condition (6) is Corollary \ref{sat-inj}.


\item  $(K, \leq_{Pure})$  where $K\subseteq R\text{-Mod}$, $K$ is closed under direct sums, direct summands and pure injective envelopes and $(K, \leq_{Pure})$ is an AEC.  Among the classes satisfying these hypothesis are: all modules, absolutely pure modules, locally injective modules, and locally pure injective modules.\footnote{In \cite[3.3]{maz2} it is explained why these classes satisfy  the hypothesis.} Conditions (1) to (3) are clear, Condition (4) is \cite[3.5]{maz2}, Condition (5) is \cite[3.8, 3.10]{maz2} and Condition (6) is basically \cite[3.13]{maz2}.




\end{enumerate}
\end{example}

\begin{remark}
The classes of Example (1) above are not contained in the classes of Example (2) and vice versa. For instance the class of flat modules is not closed under pure injective envelopes and the class of absolutely pure modules is not closed under pure epimorphic images. 

\end{remark}

We record the equivalence between superstable AECs and noetherian categories.

\begin{theorem}\label{equiv2}  Assume $\km$ satisfies Hypothesis \ref{e-h}. The following are equivalent.
\begin{enumerate}
\item $(K, \M)$ is noetherian.
\item $\K_\M$ is superstable.
\end{enumerate}
\end{theorem}

\begin{remark}
With some additional work it is possible to use the previous theorem  to give other characterizations of noetherian rings and pure semisimple rings via superstability of certain classes of modules. Since this has been done in a previous paper \cite[\S 3]{maz2} we will not go into it in this paper. 
\end{remark} 

The previous theorem shows that there is a deep connection between superstable AECs and  noetherian categories. A natural  problem is to determine how far this equivalence can be pushed.

\end{document}